\documentclass[11pt,reqno]{amsart}
\usepackage{a4wide}
\newtheorem{theorem}{Theorem}
\newtheorem*{theoremBKM}{Theorem BKM}

\newtheorem{proposition}{Proposition}
\newtheorem{lemma}{Lemma}
\newtheorem{corollary}{Corollary}
\theoremstyle{remark}
\newtheorem{remark}{Remark}

\newcommand\R{\mathbb{R}}
\newcommand\Z{\mathbb{Z}}
\newcommand\N{\mathbb{N}}

\newcommand\cB{\mathcal{B}}
\newcommand\cN{\mathcal{N}}

\newcommand\cW{\mathcal W}

\newcommand\cM{\mathcal M}
\DeclareMathOperator{\erf}{erf}
\DeclareMathOperator{\erfi}{erf^{-1}}
\newcommand{\Sph}{\mathbb{S}}

\newcommand{\cG}{\mathcal{G}}
\newcommand{\cL}{\mathcal{L}}
\newcommand{\diag}{\operatorname{diag}}
\newcommand{\cU}{{\vv U}}
\newcommand{\cUp}{{\vv U}^+}
\newcommand{\tcU}{\tilde{\vv U}}

\newcommand{\hcB}{\hat{B}}
\newcommand{\tcUp}{\tilde{\vv U}^+}
\newcommand{\wcU}{2{\vv U}}
\newcommand{\vv}[1]{{\mathbf{#1}}}
\newcommand{\bm}[1]{{\mathbf{#1}}}
\newcommand{\Rp}{\R_+}
\newcommand{\Mmn}{M_{m,n}}
\newcommand{\ve}{\varepsilon}
\newcommand{\ds}{\displaystyle}

\def\e{{\rm e}}
\def\dd{{\rm d}}

\newcommand{\hidden}[1]{}

\begin{document}

\large

\title[Diophantine Approximation and Interference Alignment]{Diophantine Approximation and applications in Interference Alignment}
\author{F. Adiceam, V. Beresnevich, J. Levesley, S. Velani}
\author{E. Zorin}
\address{Department of Mathematics, University of York,  York, YO10
5DD, UK } \email{ victor.beresnevich@york.ac.uk, jason.levesley@york.ac.uk, sanju.velani@york.ac.uk, evgeniy.zorin@york.ac.uk,  Faustin.Adiceam@york.ac.uk}
\thanks{FA research is supported by EPSRC Programme Grant: EP/J018260/1. VB and SV research is supported in part by EPSRC Programme Grant: EP/J018260/1.}

%
%

\maketitle

{\small \centerline{{\it Dedicated to  Maurice Dodson   }}
}

\begin{abstract}
This paper is motivated by recent applications of Diophantine approximation in electronics, in particular, in the rapidly developing area of Interference Alignment. Some remarkable advances in this area give substantial credit to the fundamental Khintchine--Groshev Theorem and, in particular, to its far reaching generalisation for submanifolds of a Euclidean space. With a view towards the aforementioned applications, here we introduce and prove quantitative explicit generalisations of the Khintchine--Groshev Theorem for non--degenerate submanifolds of $\R^n$. The importance of such quantitative statements is explicitly discussed in Jafar's monograph \cite[\S4.7.1]{J2010}.
\end{abstract}

\vspace{2cm}

\noindent{\small 2000 {\it Mathematics Subject Classification}\/:
Primary 11J83; Secondary 11J13, 11K60, 94A12}\bigskip

\noindent{\small{\it Keywords and phrases}\/: Metric Diophantine approximation,  Khintchine–Groshev Theorem,  non-degenerate manifolds.}

\setcounter{tocdepth}{1}
\tableofcontents

\section{Introduction}

The present paper is motivated by a recent series of publications, including \cite{GMK2010, J2010, MOGMAK, MOGMAK2, MMK2010, NW2012, WuShamaiVerdu, XU2010, ZamanighomiWang}, which utilize the theory of metric Diophantine approximation to develop new approaches in interference alignment, a concept within the field of wireless communication networks. This new link is both surprising and striking. The key ingredient from the number theoretic side is the fundamental Khintchine--Groshev Theorem and its variations. In this paper we seek to address certain problems in Diophantine approximation which crop up, or impinge upon, the applications to interference alignment. The results obtained represent quantitative refinements of the Khintchine--Groshev Theorem that are relevant to the applications mentioned above.
Indeed, the desirability of such quantitative statements is explicitly eluded to in Jafar's monograph \cite[\S4.7.1]{J2010}.

Although the main emphasis will be on the Khintchine--Groshev Theorem for submanifolds of $\R^n$,  we begin by considering the classical theory for systems of linear forms of independent variables. This approach has two benefits. Firstly, we are able to introduce the key ideas without too much technical machinery obscuring the picture. Secondly, the refinements of the classical theory produce effective results with much better constants.

In order to recall Khintchine's theorem we first define
the set  $ \cW (\psi)  $ of \emph{$\psi$-well approximable numbers}. To this end, denote by $\Rp$ the set of non--negative real numbers. Given  a real positive function  $\psi: \Rp \to \Rp $  with  $ \psi(r)\rightarrow0$ as $r\rightarrow\infty$, let then
$$
\cW (\psi):=\left\{ x \in \R :  |qx-p|<\psi(q)\text{ for i.m. }(q,p)\in\N\times\Z \right\},
$$
where `i.m.' reads `infinitely many'. For obvious reasons the function $\psi$  is often referred to as an \emph{approximating function}. The points $x$ in $\cW(\psi)$ are characterized by the property that they admit approximation by rational points $p/q$ with the error at most $\psi(q)/q$.

A simple `volume' argument together with the Borel--Cantelli Lemma from pro\-ba\-bi\-li\-ty theory implies that
\[
|\cW (\psi)|= 0   \quad {\rm \ if \  }  \quad \sum_{q=1}^\infty  \,\,\psi(q)  \;   <\infty  \,,
\]
where $|X|$ stands for the Lebesgue measure of $X\subset\R$. The above convergence statement represents the easier part of the following beautiful result due to Khintchine which gives a criterion for the size of the set $\cW(\psi)$ in terms of Lebesgue measure. In what follows, we say that $X\subset\R$ is \emph{full} in $\R$ and write $|X| = \text{\rm F\small{ULL}}$ if $|\R \setminus X | = 0 $; that is, the complement of $X$ in $\R$ is of Lebesgue measure zero. The following is a slightly more general version of Khintchine, see \cite{Beresnevich-Dickinson-Velani-06:MR2184760}.

\medskip

\renewcommand{\thetheorem}{\Alph{theorem}}

\begin{theorem}[Khintchine, 1924]\label{thmA}
Let $\psi$ be an approximating function.
Then
\[
    |\cW(\psi)| \ = \ \begin{cases} 0
      &\ds\text{if } \quad \sum_{q=1}^\infty  \psi(q)<\infty \, , \\[3ex]
      \text{\rm F\small{ULL}}
      &\ds\text{if } \quad  \sum_{q=1}^\infty  \psi(q)=\infty
      \  \text{ and $\psi$ is
                    monotonic}.
                  \end{cases}
\]
\end{theorem}

\medskip

Thus, given any monotonic approximating function $\psi$, for almost all\footnote{\,`For almost all' means for all except from a set of Lebesgue measure zero.} $x \in\R$ the inequality $|x-p/q | < \psi(q)/q$  holds for infinitely many rational numbers $p/q$  if and only if the
sum $\sum_{q=1}^\infty\psi(q)$ diverges.

There are various generalisations of Khintchine's theorem to higher dimensions --- see \cite{BBDV09} for an overview. Here we shall consider the case of systems of linear forms which originates from a paper by Groshev in 1938.
In what follows, $m$ and $n$ will denote positive integers and $\Mmn$ will stand for the set of $m\times n$ matrices over $\R$. Given a function $\Psi:\Z^n\to\Rp$, let
\[
\cW_{m,n}(\Psi):=\left\{ \vv X=(x_{i,j}) \in \Mmn :  \|\vv X\vv a\|<\Psi(\vv a)\text{ for i.m. }\vv a\in\Z^n\setminus\{\vv0\}\right\},
\]
where $\vv a=(a_1,\dots,a_n)$,
$$
\|\vv X\vv a\| :=  \max_{1\le i\le m}\|x_{i,1}a_1 + \ldots + x_{i,n}a_n\|
$$
and $ \| x\| :=  \min \{ | x - k | : k \in \Z \}$  is the distance of $ x \in \R$ from the nearest integer.
Given a subset $X$ in $\Mmn$, we will write $|X|_{mn}$ for its ambient (i.e.~$mn$--dimensional) Lebesgue measure. It is easily seen that $\cW_{1,1}(\Psi)$ coincides with $\cW(\psi)$ when $\Psi(q)=\psi(|q|)$. Therefore the following result is the natural extension of Theorem~\ref{thmA} to higher dimensions. Notice that there is no monotonicity assumption on the approximating function.

\begin{theorem}\label{thmB}
Let $m,n\in\N$ with $nm>1$, $\psi:\N\to\Rp$ be an approximating function and
\begin{equation} \label{theorem_one_def_S}
\Sigma_\psi:=\sum_{q=1}^\infty  q^{n-1}\psi(q)^m\,.
\end{equation}
Let $\Psi:\Z^n\to\Rp$ be given by $\Psi(\vv a):=\psi(|\vv a|)$ for $\vv a=(a_1,\dots,a_n)\in\Z^n\setminus\{\vv0\}$, where $|\vv a|=\max_{1\le i\le n}|a_i|$. Then
\[
    |\cW_{m,n}(\Psi)|_{mn} \ = \ \begin{cases} 0
      &\ds\text{if } \quad \Sigma_\psi<\infty \, , \\[1ex]
      \text{\rm F\small{ULL}}
      &\ds\text{if } \quad  \Sigma_\psi=\infty\,.
                  \end{cases}
\]
\end{theorem}

Theorem~B was first obtained by Groshev  under the assumption that $q^n\psi(q)^m$ is  monotonic in the case of divergence. The redundancy of the monotonicity condition for $n\ge 3$ follows from Schmidt's paper \cite[Theorem~2]{Schmidt-1960} and for $  n = 1 $ from Gallagher's paper \cite{Ga65}. Theorem~B as stated was eventually proved in \cite{BV10} where the remaining case of $n=2$ was addressed. The convergence case of Theorem~B is a relatively simple application of the Borel--Cantelli Lemma and it holds for arbitrary functions $\Psi$.  Thus together with  Theorem A,   we have the following extremely general statement in the case of convergence.

\begin{theorem}\label{thmC}
Let $m,n\in\N$ and $\Psi:\Z^n\to\Rp$ be any function such that the sum
\begin{equation} \label{S}
\Sigma_\Psi:=\sum_{\vv a\in\Z^n\setminus\{\vv 0\}}\Psi(\vv a)^m
\end{equation}
converges. Then
$$
    |\cW_{m,n}(\Psi)|_{mn}  =  0\,.
$$
\end{theorem}

An immediate consequence of Theorem~\ref{thmC} is the following statement.

\begin{corollary}\label{cor1}
Let  $\Psi$ be as in Theorem~\ref{thmC}. Then, for almost every $\vv X \in \Mmn$ there exists a constant $\kappa (\vv X)  >0 $  such that
\begin{equation}\label{vb1}
 \|\vv X\vv a \| \ >  \  \kappa(\vv X)  \,   \Psi(\vv a)  \qquad   \forall  \   \  \vv  a \in \Z^n\setminus\{\vv0\}\,.
\end{equation}
\end{corollary}

In recent years estimates of this kind have become an important ingredient in the study of the achievable number of degrees of freedom in various schemes on Interference Alignment from electronics communication --- see, e.g., \cite{MOGMAK}. The applications typically require that $\kappa(\vv X)$ is independent of $\vv X$. Unfortunately, this is impossible to guarantee with probability 1, that is on a set of full Lebesgue measure. To demonstrate this claim, let us define the following set~:
\begin{equation} \label{def_B}
 \mathcal{B}_{m,n}(\Psi,\kappa):=\Big\{\vv X \in \Mmn : \|\vv X\vv a\|  > \kappa\Psi(\vv a)  \ \ \ \forall\  \vv a \in \Z^n\setminus\{\vv0\}\Big\}\,.
\end{equation}
Then, for any $\kappa$ and $\Psi$, the set $\mathcal{B}_{1,n}(\Psi,\kappa)$ will not contain
$$
[-\kappa\Psi(\vv a),\kappa\Psi(\vv a)]\times\R^{n-1}
$$
with $\vv a=(1,0,\dots,0)$. This set is of positive probability.
In the light of this example it becomes highly desirable to address the following problem~:

\medskip

\noindent\textbf{Problem.} \textit{Investigate the dependence between $\kappa$ and the probability of $\mathcal{B}_{m,n}(\Psi,\kappa)$.}

\medskip

As the first step to understanding this problem we obtain the following straightforward consequence of Theorem~C.

\renewcommand{\thetheorem}{\arabic{theorem}}
\setcounter{theorem}{0}

\begin{theorem} \label{theorem_linear_forms}
Let $m,n\in\N$ and $\mu$ be a probability measure on $\Mmn$ that is absolutely continuous with respect to Lebesgue measure on $\Mmn$. Let $\Psi:\Z^n\to\Rp$ be any function such that \eqref{S} converges. Then for any $\delta\in(0,1)$ there is a constant $\kappa>0$ depending only on $\mu$, $\Psi$ and $\delta$ such that
\begin{equation} \label{theorem_linear_forms_ie_statement}
\mu\left(\mathcal{B}_{m,n}(\Psi,\kappa)\right)\ge 1-\delta.
\end{equation}
\end{theorem}

Prior to giving a proof of this theorem recall that a measure $\mu$ on $\Mmn$ is \emph{absolutely continuous with respect to Lebesgue measure} if there exists a Lebesgue integrable function $f:\Mmn\rightarrow\Rp$ such that for every Lebesgue measurable subset $A$ of $\Mmn$, one has that
\begin{equation}\label{vb2}
\mu(A)=\int_Af,
\end{equation}
where $\int_A{}f$ is the Lebesgue integral of $f$ over $A$. The function $f$ is often referred to as the \emph{distribution (or density) of $\mu$}.

\begin{proof}
Since $\mu$ is absolutely continuous with respect to Lebesgue measure,  Theorem~C implies that $\mu(\cW_{m,n}(\Psi))=~0$. Hence $\mu(M_{m,n}\setminus\cW_{m,n}(\Psi))=\mu(M_{m,n})=1$. Note that
$$
\bigcup_{\kappa>0}\mathcal{B}_{m,n}(\Psi,\kappa)=M_{m,n}\setminus\cW_{m,n}(\Psi)\,.
$$
Theorem~\ref{theorem_linear_forms} now  follows on using the continuity of measures. \end{proof}

In view of our previous discussion we have that $\kappa\to0$ as $\delta\to0$. Then, the above problem specialises to the explicit understanding of the dependence of $\kappa$ on $\delta$. This will be the main content of  the next section. Subsequent sections will be devoted to obtaining similar effective version of the convergence Khintchine--Groshev Theorem for non--degenerate submanifolds of $\R^n$.  This constitutes the main substance of the paper. The results are obtained by exploiting the techniques of Bernik, Kleinbock and Margulis \cite{Bernik-Kleinbock-Margulis-01:MR1829381} originating from the seminal work of Kleinbock and Margulis \cite{Kleinbock-Margulis-98:MR1652916} on the Baker--Sprind\v{z}uk conjecture.

\section{The theory for independent variables}

To begin with we give an alternative proof of Theorem~\ref{theorem_linear_forms} which introduces an explicit construction that will be utilized for quantifying the dependence of $\kappa$ on $\delta$. Indeed, in the case that $\mu$ is a uniform distribution on a unit cube the proof already identifies the required dependence.

\subsection{Theorem~\ref{theorem_linear_forms} revisited}\label{2.1}

By a unit cube in $\Mmn$ we will mean a subset of $\Mmn$ given by
$$
\big\{(x_{i,j})\in\Mmn:\alpha_{i,j}\le x_{i,j}<\alpha_{i,j}+1\big\}
$$
for some fixed matrix $(\alpha_{i,j})\in\Mmn$.
Given $\vv a\in\Z^n\setminus\{\vv0\}$ and $\ve>0$, let $\mathcal{W}(\vv a,\ve)$ denote the set of $\vv X\in\Mmn$ such that
\begin{equation} \label{theorem_linear_forms_reciproque_inequality}
 \|\vv X\vv a\| \ \le  \ve\,.
\end{equation}
It is easily seen that $\mathcal{W}(\vv a,\ve)$ is invariant under additive translations by an integer matrix; that is,
$$
\mathcal{W}(\vv a,\ve)+\vv B=\mathcal{W}(\vv a,\ve)
$$
for any $\vv B\in\Mmn(\Z)$, where $\Mmn(\Z)$ denotes the set of $m\times n$ matrices with integer entries. Furthermore, we have that
\begin{equation}\label{v103}
|\mathcal{W}(\vv a,\ve)\cap P|_{mn}=(2\ve)^m
\end{equation}
for any $0\le \ve\le\tfrac12$ and any unit cube $P$ in $\Mmn$. This follows, for example, from \cite[Chapter~1, Lemma~8]{Sprindzuk-1979-Metrical-theory}.
Then,  since
\begin{equation}\label{sv111} \Sigma_\Psi:=\sum_{\vv a\in\Z^n\setminus\{\vv 0\}}\Psi(\vv a)^m <\infty  \, \end{equation}
we must have that
\begin{equation}\label{v105}
M_\Psi:=\sup\{\Psi(\vv a):\vv a\in\Z^n\setminus\{\vv0\}\}<\infty.
\end{equation}
In what follows we will assume that
\begin{equation}\label{v104}
2\kappa M_\Psi\le 1\,.
\end{equation}
This condition ensures that we can apply \eqref{v103} with $\ve=\kappa\Psi(\vv a)$.

Fix a unit cube $P_{\vv0}$ in $\Mmn$ and for each $\Delta\in\Mmn(\Z)$, let
$$
P_\Delta:=P_{\vv0}+\Delta
$$
denote  the additive translation of $P_{\vv0}$ by $\Delta$. Clearly, $P_\Delta$ itself is a unit cube. Furthermore,
\begin{equation}\label{v111}
\Mmn=\bigcup_{\Delta\in\Mmn(\Z)}^{\circ} P_\Delta \, .
\end{equation}
Note that the union is disjoint. Using \eqref{v103} and the fact that
$$
\Mmn\setminus \mathcal{B}_{m,n}(\Psi,\kappa)=\bigcup_{\vv a\in\Z^n\setminus\{\vv 0\}}\mathcal{W}(\vv a,\kappa\Psi(\vv a))  \, ,
$$
we obtain that for each $\Delta\in\Mmn(\Z)$,
\begin{eqnarray}\label{v100}
|P_{\Delta}\setminus \mathcal{B}_{m,n}(\Psi,\kappa)|_{mn} & \le  &  \sum_{\vv a\in\Z^n\setminus\{\vv 0\}}|\mathcal{W}(\vv a,\kappa\Psi(\vv a))\cap P_{\Delta}|_{mn}   \nonumber  \\[2ex]
& =  &  \sum_{\vv a\in\Z^n\setminus\{\vv 0\}}(2\kappa\Psi(\vv a))^m
\ =   \ (2\kappa)^m\Sigma_\Psi\,.
\end{eqnarray}
Since $\mu$ is a probability measure, it follows from \eqref{v111} that there exists a finite subset $A\subset\Mmn(\Z)$ such that
\begin{equation}\label{v106}
\mu\left(\bigcup_{\Delta\in A} P_{\Delta}\right)>1-\delta/2\,.
\end{equation}
Let $N=\#A$ be the number of elements in $A$. Since $\mu$ is absolutely continuous with respect to Lebesgue measure, for every $\Delta\in A$ and any $\ve_1>0$, there exists $\ve_2$ such that for any measurable subset $X$ of $P_\Delta$,
\begin{equation}\label{v101}
|X|_{mn}<\ve_2\quad\Rightarrow\quad \mu(X)<\ve_1\,.
\end{equation}
In view of \eqref{v100}, applying \eqref{v101} to $X=P_{\Delta}\setminus \mathcal{B}_{m,n}(\Psi,\kappa)$ and $\ve_1=\delta/(2N)$ implies the existence of
$$
\ve_2=\ve_2(\Delta,\delta,N)>0
$$
such that
\begin{equation}\label{v102}
\mu(P_{\Delta}\setminus \mathcal{B}_{m,n}(\Psi,\kappa))<\delta/(2N)\qquad\text{if}\qquad(2\kappa)^m\Sigma_\Psi\le \ve_2(\Delta,\delta,N)\,.
\end{equation}
In particular, the second inequality in~\eqref{v102} holds if
$$
\kappa\le \kappa_\Delta:=\frac12\left(\frac{\ve_2(\Delta,\delta,N)}{\Sigma_\Psi}\right)^{1/m}\,.
$$
Since $A$ is finite, there exists  $\kappa$ satisfying \eqref{v104} and
$$
0<\kappa\le \min_{\Delta\in A} \kappa_\Delta\,.
$$
Clearly, for such a choice of $\kappa$ the first inequality in~\eqref{v102} holds for any $\Delta\in A$. Hence, by \eqref{v106} and the additivity of $\mu$ we obtain  that
\begin{eqnarray*}
\mu(\Mmn\setminus \mathcal{B}_{m,n}(\Psi,\kappa))  & \le  &  \frac\delta2+
\sum_{\Delta\in A}\mu(P_{\Delta}\setminus \mathcal{B}_{m,n}(\Psi,\kappa))
\\[1ex]
& \le  &  \frac\delta2+
\sum_{\Delta\in A}\frac\delta{2N}  \ =  \ \frac\delta2+N\frac\delta{2N} \ =\ \delta\,.
\end{eqnarray*}
The upshot of this is that
\begin{equation}\label{v108}
\mu(\mathcal{B}_{m,n}(\Psi,\kappa)) = 1-\mu(\Mmn\setminus \mathcal{B}_{m,n}(\Psi,\kappa))\ge1-\delta\, ,
\end{equation}
which completes the proof of Theorem~\ref{theorem_linear_forms}.

\medskip

\subsection{Quantifying the dependence of $\kappa$ on $\delta$ \label{explicitind}}

We now turn our attention to quantifying the dependence of $\kappa$ on $\delta$ within the context of Theorem  \ref{theorem_linear_forms}.  To this end, we will make use of the  $L_p$ norm.  Given a Lebesgue measurable function $f:\Mmn\rightarrow\Rp$, a measurable subset $X$ of $\Mmn$ and $p\ge1$, we write $f\in L_p(X)$ if the Lebesgue integral
$$
\int_X |f|^p:=\int_{\Mmn}|f|^p\chi_X
$$
exists and is finite.  Here $\chi_X$ is the characteristic function of $X$. For  $f\in L_p(X)$, the $L_p$ norm of $f$ on $X$ is defined by
$$
\|f\|_{p,X} \, :=  \, \left(\int_X |f|^p\right)^{1/p}\,.
$$
In the case $p=\infty$, the $L_\infty$--norm on $X$ is defined as the essential maximum of $|f|$ on $X$; that is,
\[
\|f\|_{\infty,X}:=\inf\left\{c\in\R: |f(x)|\le c\text{ for almost all }x\in X\right\}\,.
\]
If $\|f\|_{\infty,X}<\infty$,  then we write $f\in L_\infty(X)$. For example, if $f$ is continuous and $X$ is a non--empty open subset of $\Mmn$, then $\|f\|_{\infty,X}$ is simply the supremum of $f$ on $X$.  The following lemma gathers together two well know facts regarding the  $L_p$ norm.

\begin{lemma}\label{l1}\ \\
\begin{enumerate}
  \item ~~~For any $p\ge1$ and any measurable subsets $X\subset Y$,  $$\|f\|_{p,X}\le \|f\|_{p,Y}\,.$$ \\[-4ex]
  \item ~~~$($H\"older's inequality$)$\,  For any $1\le p,q\le\infty$ satisfying $\frac1p+\frac1q=1$,
  $$
  \left|\int_Xfg\right|\le\|f\|_{p,X}\|g\|_{q,X}\,.
  $$
\end{enumerate}
\end{lemma}

\vspace*{3ex}

 The next lemma is  a corollary of Lemma  \ref{l1}.

\begin{lemma} \label{lemma_passage_Lebesgue_to_nu_via_v}
Let $p>1$ and $\mu$ be a probability measure on $\Mmn$ with density $f$. Let $X$ be a Lebesgue measurable subset of $\Mmn$. If $f\in L_p(X)$,  then
$$
\mu(X)\le \|f\|_{p,X}|X|_{mn}^{1-1/p}\,.
$$
\end{lemma}

\begin{proof}
By definition, we have that
$$
\mu(X)=\int_X f \,.
$$
Define $q$ by the equation $\tfrac1p+\tfrac1q=1$. Then by H\"older's inequality, we have that
$$
\mu(X)=\int_X f\times 1  \ \le  \ \|f\|_{p,X}\|1\|_{q,X}  \ =\  \|f\|_{p,X}\left(\int_X1^q \right)^{1/q}
\ \le  \  \|f\|_{p,X}|X|_{mn}^{1-1/p}
$$
as required.
\end{proof}

We are now in the position to provide an effective version of Theorem~\ref{theorem_linear_forms}.
Let $P_{\vv0}$ and $A$ be the same as in \S\ref{2.1}. In particular, assume that \eqref{v106} holds. Furthermore,  assume  that there exists  some $p>1$ such that for every $\Delta\in A$,  the density $f$ of $\mu$ has finite $L_p$ norm on $P_\Delta$ .

\noindent Let $\kappa$ be such that \eqref{v104} is satisfied. In this case, \eqref{v100} holds for every $P_\Delta$ with $\Delta\in A$. By Lemmas~\ref{l1} and \ref{lemma_passage_Lebesgue_to_nu_via_v},
$$
\mu(P_{\Delta}\setminus \mathcal{B}_{m,n}(\Psi,\kappa))\le\|f\|_{p,P_\Delta}\cdot|P_{\Delta}\setminus \mathcal{B}_{m,n}(\Psi,\kappa)|_{mn}^{1-1/p}
\,.
$$
Using \eqref{v100},  we obtain that
\begin{equation}\label{v107}
\mu(P_{\Delta}\setminus \mathcal{B}_{m,n}(\Psi,\kappa))\le\|f\|_{p,P_\Delta}\cdot\Big((2\kappa)^m\Sigma_\Psi\Big)^{1-1/p}
\,
\end{equation}
where $\Sigma_\Psi$ is given by \eqref{sv111}.
It follows that
\begin{eqnarray*}
\mu(\Mmn\setminus \mathcal{B}_{m,n}(\Psi,\kappa))  & \le  &  \frac\delta2+
\sum_{\Delta\in A}\mu(P_{\Delta}\setminus \mathcal{B}_{m,n}(\Psi,\kappa))  \\[2ex]
& \le & \frac\delta2+
\Big((2\kappa)^m\Sigma_\Psi\Big)^{1-1/p}  \Sigma_f  \ \le  \  \delta
\end{eqnarray*}
if
$$
\kappa\le \frac12\left(\Sigma_\Psi^{-1}\left(\frac{\delta}{2\Sigma_f}\right)^{\frac p{p-1}}\right)^{1/m}\,,
$$
where
\begin{equation}\label{vb300}
\Sigma_f:=\sum_{\Delta\in A}\|f\|_{p,P_\Delta} \, .
\end{equation}
Since $A$ is finite, the quantity  $\Sigma_f$ is also  finite.
The upshot of the above discussion is the following statement.

\begin{theorem}[Effective version of Theorem~\ref{theorem_linear_forms}] \label{theorem_linear_forms+}
Let $m,n\in\N$, $\mu$, $\Psi$ be as in Theorem~\ref{theorem_linear_forms}, let $M_\Psi$ be given by \eqref{v105} and let $f$ denote the density of $\mu$. Furthermore,  let $P_{\vv0}$ be any unit cube in $\Mmn$ and $A$ be any finite subset of $\Mmn(\Z)$ satisfying \eqref{v106}. Assume there exists  $p>1$ such that  $f\in L_p(P_\Delta)$ for any $\Delta\in A$ and also assume that the quantity $\Sigma_f$ is given by \eqref{vb300}.
Then, for any $\delta\in(0,1)$, inequality~\eqref{theorem_linear_forms_ie_statement} holds with
\begin{equation} \label{theorem_linear_forms_def_kappa_one}
\kappa:=\frac12\min\left\{\frac{1}{M_\Psi},\ \left(\Sigma_\Psi^{-1}\left(\frac{\delta}{2\Sigma_f}\right)^{\frac p{p-1}}\right)^{1/m}
\right\}\,.
\end{equation}
In this formula, the quotient $p/(p-1)$ should be taken as equal to 1 when $p=\infty$.
\end{theorem}

\begin{remark}
In the case when $\Psi$ is even, that is $\Psi(-\vv a)=\Psi(\vv a)$ for all $\vv a\in\Z^n\setminus\{\vv0\}$, one can improve formula \eqref{theorem_linear_forms_def_kappa_one} for $\kappa$ by replacing $\Sigma_\Psi$ with $\tfrac12\Sigma_\Psi$. This is an obvious consequence of the fact that in this case the sets $\mathcal{W}(\vv a,\kappa\Psi(\vv a))$ and $\mathcal{W}(-\vv a,\kappa\Psi(-\vv a))$ coincide and therefore do not have to be counted twice within the proof.
\end{remark}

There are various simplifications and specialisations of Theorem~\ref{theorem_linear_forms+} when we have extra information regarding the measure $\mu$. The following is a natural corollary which is particularly relevant for probability measures $\mu$ with bounded distribution $f$ and mean value about the  origin.

\begin{corollary}\label{cor2}
Let $m,n\in\N$, $\mu$, $\Psi$, $M_\Psi$ be  as in Theorem~\ref{theorem_linear_forms} and let the density $f$ of $\mu$ be bounded above by a constant $K>0$. Furthermore,  let $T$ be the smallest positive integer such that
\begin{equation} \label{theorem_one_defN}
\mu\left([-T,T)^{mn}\right)\ge 1-\delta/2.
\end{equation}
Then,  for any $\delta\in(0,1)$, inequality \eqref{theorem_linear_forms_ie_statement} holds with
\begin{equation}\label{v110}
\kappa:=\frac12\min\left\{\frac{1}{M_\Psi},\ \left(\frac{\delta}{2K(2T)^{mn}\Sigma_\Psi}\right)^{1/m}
\right\}\,.
\end{equation}
\end{corollary}

\begin{proof}
With respect to  Theorem~\ref{theorem_linear_forms+},  let  $p=\infty$ and $A$ be the collection of cubes $P_{\Delta}$ that  exactly tiles $[-T,T)^{mn}$. Then, $\#A=(2T)^{mn}$ and  thus
$
\Sigma_f\le K(2T)^{mn}
$.
Now,   \eqref{v110} trivially  follows from \eqref{theorem_linear_forms_def_kappa_one}.
\end{proof}

\subsection{Numerical examples}

In what follows, we will use the standard Gaussian error function
\[
\erf(x):=\frac{1}{\sqrt{2\pi}}\int_{-\infty}^{x}\e^{-t^2/2}\dd t\,.
\]
It is readily verified that the function  $\erf$ is continuous, strictly increasing and that
\begin{align*}
\lim_{x\rightarrow-\infty}\erf(x)=0\qquad\text{and}\qquad
\lim_{x\rightarrow+\infty}\erf(x)=1.
\end{align*}
As usual, for $0<y<1$, define $\erfi(y)$ to be the unique value $x\in\R$ such that $\erf(x)=y$.
Furthermore, define formally $\erfi(0):=-\infty$ and $\erfi(1):=+\infty$.

Consider now  Corollary~\ref{cor2} in the case when $m=n=1$ and when $\mu$ follows the standard Gaussian distribution $\cN(0,1)$. It can then be verified that  Corollary~\ref{cor2} implies that inequality~\eqref{theorem_linear_forms_ie_statement} holds with
\begin{equation} \label{theorem_one_example_one_kappa_one}
\kappa=\frac{\delta\sqrt{2\pi}}{8\cdot N\cdot \Sigma_\Psi}  \, ,
\end{equation}
where $N:=\lceil\erfi\left(1-\delta/4\right)\rceil$. Here $\lceil x\rceil$ is the ``ceiling'' of $x$, that is the smallest integer that is bigger than or equal to $x\in\R$.  We now consider explicit  approximating functions. First, let  $\Psi$ be the function given by $\Psi(q)=0$ if $q\leq 0$, $$\Psi(q):=\frac{1}{2 q \cdot\log^2q}   \quad  {\rm if} \quad   q \ge 2 \quad  {\rm and } \quad  \Psi(1):=1/2  \, . $$   Then $\Sigma_\Psi<1.555$ and on making use of
 \eqref{theorem_one_example_one_kappa_one}  we obtain the following table for  values of  $N$ and $\kappa$~:

\begin{center}
\begin{tabular}{|c|c|c|c|c|c|c|}
\hline
$\delta$ & $0.5$ & $0.25$ & $0.1$ & $0.01$ & $10^{-3}$ & $10^{-5}$\\
\hline
$N$ & 2 & 2 & 3 & 4 & 4 & 5\\ \hline
$\kappa$ & $0.05$ & $0.025$ & $0.0067$ & $5\cdot 10^{-4}$ & $5\cdot 10^{-5}$ & $4\cdot 10^{-7}$ \\ \hline
\end{tabular}
\end{center}
It follows for instance from this set of data  that for $99\%$ of the values of the random variable $x$ with normal distribution $\mathcal{N}(0,1)$, one has that
\[
 \|q x \| \ >  \ \frac{1}{2000}\cdot\Psi(q)  \qquad {\rm \ for \ all \ }   \   q \in \N.
\]

In the next example, we fix a $Q\in\N$ and consider the approximating function $\Psi$ given by
$$
\Psi(q):=\begin{cases}
\frac{1}{Q} & \text{ if } 1\le q\le Q,\\
0 & \text{ otherwise}.
\end{cases}
$$
Then $\Sigma_\Psi=1$ and one can readily verify that
\begin{enumerate}
\item[(i)] for at least $75\%$ of the values of the random variable $x$ with normal distribution $\mathcal{N}(0,1)$, one has that
\[
\| qx \| > \frac{1}{13 Q}  \qquad {\rm for \ all \ }   \   q \in [-Q,Q],\, q\neq 0,
\] \\[-4ex]
\item[(ii)] for at least $90\%$ of the values of the random variable $x$ with normal distribution $\mathcal{N}(0,1)$, one has that
$$
\| qx \| > \frac{1}{50 Q}  \qquad {\rm for \ all \ }   \   q \in [-Q,Q],\, q\neq 0 .
$$
\end{enumerate}

\section{Diophantine approximation on manifolds}

The aim is to establish an analogue of Theorem~\ref{theorem_linear_forms+} for submanifolds $\cM$ of $\R^n$. More precisely, we consider the set $\cB_{n}(\Psi,\kappa)\cap\cM$, where
 $$
 \cB_{n}(\Psi,\kappa):=\cB_{1,n}(\Psi,\kappa)\,.
 $$
 The fact that the points  of interest are of dependent variables, reflecting the fact that they lie on $ \cM$,
introduces major difficulties in attempting to describe the measure
theoretic structure of $\cB_{n}(\Psi,\kappa)\cap\cM$.

\medskip

\noindent{\em Non--degenerate manifolds. } In order to make any
reasonable progress with the above problems it is not unreasonable
to assume that the manifolds $ \cM$ under consideration are
{\em non--degenerate}. Essentially, these are smooth
submanifolds of $\R^n$ which are sufficiently curved so as to
deviate from any hyperplane. Formally, a  manifold $\cM$ of
dimension $d$ embedded in $\R^n$ is said to be \emph{non--degenerate} if it
arises from a non--degenerate map $\vv f: \vv U\to \R^n$ where $\vv U$ is an
open subset of $\R^d$ and $\cM:=\vv f(\vv U)$. The map $\vv f: \vv U\to
\R^n,\vv x\mapsto \vv f(\vv x)=(f_1(\vv x),\dots,f_n(\vv x))$ is said to be
\emph{$l$--non--degenerate at} $\vv x\in \vv U$, where $l\in\N$,
if $\vv f$ is $l$ times continuously differentiable on some
sufficiently small ball centred at $\vv x$ and the partial derivatives
of $\vv f$ at $\vv x$ of orders up to $l$ span $\R^n$. The map $\vv f$ is
\emph{non--degenerate} at $\vv x$ if it is $l$--non--degenerate at $\vv x$ for some $l\in\N$. As is well known, any
real connected analytic manifold not contained in any hyperplane of
$\R^n$ is non--degenerate at every point \cite{Kleinbock-Margulis-98:MR1652916}.

Observe that if the dimension of the manifold $\cM$ is strictly less than $n$
then we have that $|\cB_n(\Psi,\kappa)\cap\cM|_n=0$ irrespective of the
approximating function $\Psi$ and $\kappa$. Thus, when referring to the Lebesgue
measure of the set $ \cB_n(\Psi,\kappa)\cap\cM $ it is always  with
reference to the induced Lebesgue measure on $\cM$. More
generally, given a subset $S$ of $\cM$ we shall write
$|S|_{\cM} $ for the  measure of $S$ with respect to the induced
Lebesgue measure on $\cM$. Without loss of generality, we will assume that $|\cM|_\cM=1$ as otherwise the measure can be re--normalized accordingly.

The following statement is a straightforward  consequence of the main result of Bernik, Kleinbock and Margulis in~\cite{Bernik-Kleinbock-Margulis-01:MR1829381}.

\begin{theoremBKM}
Let $\cM$ be a non--degenerate submanifold of $\R^n$. Let $\Psi:\Z^n\to\Rp$ be monotonically decreasing in each variable and such that
\begin{equation} \label{def_S_Psi}
\Sigma_\Psi:=\sum_{\vv q\in\Z^n\setminus\{0\}}\Psi(\vv q)<\infty\,.
\end{equation}
Then, for any $\delta\in(0,1)$, there is a constant $\kappa>0$ depending on $\cM$, $\Sigma_\Psi$ and $\delta$ only such that

\begin{equation} \label{theorem_linear_forms_ie_statement_v2}
|\cB_n(\Psi,\kappa)\cap\cM|_\cM\ge 1-\delta.
\end{equation}
\end{theoremBKM}

\bigskip

\begin{remark}
Theorem~BKM holds for arbitrary probability measures supported on $\cM$  that are  absolutely continuous with respect to the induced Lebesgue measure on $\cM$, thus giving an analogue of Theorem~\ref{theorem_linear_forms+} for manifolds. As in the case of Theorem~\ref{theorem_linear_forms}, the more general result follows from the Lebesgue statement.
\end{remark}

It is worth pointing out that the main result in~\cite{Bernik-Kleinbock-Margulis-01:MR1829381} actually implies that the union $\bigcup_{\kappa>0}\cB_{n}(\Psi,\kappa)\cap\cM$ has full measure on $\cM$. Theorem BKM as stated above follows from~\cite[Theorem~1.1]{Bernik-Kleinbock-Margulis-01:MR1829381}\footnote{Throughout, results and page numbers within \cite{Bernik-Kleinbock-Margulis-01:MR1829381} are with reference to the arXiv version: math/0210298v1} on using the continuity of measures. Our main goal is to quantify the dependence of $\kappa$ on $\delta$. Theorem~\ref{theorem_expicit_1_1} of \S\ref{6} below explicitly quantifies this dependence.   However, the statement is rather technical and we prefer to state for now a cleaner result that shows that the dependency between $\kappa$ and $\delta$ is polynomial.

\begin{theorem} \label{soft}
Let $l\in\N$ and let $\cM$ be a compact $d$--dimensional $C^{l+1}$ submanifold of $\R^n$ that is $l$--non--degenerate at every point.
Let $\mu$ be a probability measure supported on $\cM$ absolutely continuous with respect to $|\, .\, |_{\cM} $. Let $\Psi:\Z^n\to\Rp$ be a monotonically decreasing function in each variable satisfying \eqref{def_S_Psi}.
Then there exist positive constants $\kappa_0,C_0,C_1$ depending on $\Psi$ and $\cM$ only such that for any $0<\delta<1$, the inequality
\begin{equation} \label{theorem_linear_forms_ie_statement_v2++}
\mu(\cB_n(\Psi,\kappa)\cap\cM)\ge 1-\delta
\end{equation}
holds with
\begin{equation}\label{kappavb}
\kappa:=\min\left\{\kappa_0,\ C_0\Sigma_\Psi^{-1}\delta,\ C_1\delta^{d(n+1)(2l-1)}\right\}\,.
\end{equation}
\end{theorem}

\section{Diophantine approximation on manifolds and wireless technology}

In short, interference alignment is a linear precoding technique that attempts to align signals in time, frequency, or space.  The following exposition is an attempt to illustrate at a basic level the role of Diophantine approximation in implementing this technique. We stress that this section is not meant for the ``electronics'' experts.  We consider two examples.  The first basic example  brings into play the theory of Diophantine approximation while the second slightly more complicated example also  brings into play the manifold theory.


\medskip

\noindent {\small EXAMPLE} 1. \, There are two people (\emph{users})  $S_1$ and $S_2$ who wish to send (\emph{transmit}) a  message  (\emph{signal}) $u \in \{ 0,1\} $ and $v \in \{ 0,1\}$ respectively  along a single communication channel (could be a cable or radio channel) to a person (\emph{receiver}) $R$. Suppose there is a certain degree of fading (\emph{channel coefficients}) associated with the messages during transmission along the channel.  This for instance could be dependent on the distance of the users to the receiver and in the case of a radio channel, the reflection caused by obstacles such as  buildings in the path of the signal. It is worth stressing that this aspect of ``fading'' associated with a signal  should not be confused with the more familiar aspect of a signal being corrupted by ``noise''  that will be discussed  a little later.    Let $h_1$ and $h_2$  denote the fading factors associated with the messages being sent by  $S_1$ and $S_2$ respectively. These are strictly positive numbers and assume their sum is one.  Also, assume that the channel is additive.  That is to say that  $R$  receives the  message:
\begin{equation} \label{nonoise}
y = h_1 x_1  +  h_2 x_2   \qquad {\rm where \  }  \  \   x_1= u    \quad {\rm and } \quad x_2= v  \, .
\end{equation}
Specifically, the outcomes of $y$ are
\begin{equation}\label{nonoiseout}
y=\begin{cases}
0 &\text{ if  } \quad  u=v=0 \\[1ex]
h_1 &\text{ if  }\quad  u=0 \text{ and } v=1\\[1ex]
h_2 &\text{ if  } \quad u=1 \text{ and } v=0\\[1ex]
1= h_1+h_2  &\text{ if  }  \quad u=v=1 \, ,
\end{cases}
\end{equation}
and if $h_1 \neq h_2$, the receiver is obviously able to recover the messages $u$ and $v$. Moreover, the greater the mutual separation of the above four outcomes in the unit interval $I=[0,1]$,  the better the tolerance for error (\emph{noise}) during the transmission of the signal. The noise can be a combinations of various factors but often the largest contributing factor is  the  interference caused by other communication channels.   If  $z$ denotes the noise, then instead of \eqref{nonoise}, in  practice $R$  receives the  message:
\begin{equation}
\label{withnoise}
y = h_1 x_1  +  h_2 x_2  + z     \qquad {\rm where \  }  \  \   x_1= u    \quad {\rm and } \quad x_2= v  \, .
\end{equation}

%
%
%
%
%
%
%
%
%
%
%
%
%
%

Now let $d$ denote the minimum distance between the four outcomes of $y\in I$ which are explicitly given by \eqref{nonoiseout}. Then as long as the absolute value $|z|$ of the noise is strictly less than $d/2$, the receiver is  able to recover the messages $u$ and $v$. This is simply due to the fact that intervals of radius  $d/2$ centered at the four outcomes of $y$ are disjoint.  In this basic example, it is easy to see that the maximum separation between the four outcomes is attained when $h_1 = 1/3$ and $h_2 = 2/3$.   In this case $d = 1/3$,  and we are able to recover the messages $u$ and $v$ as long as $|z| < 1/6$.  The upshot is that the closer the real numbers $h_1$ and $h_2$ are to $1/3$ and $2/3$ the better the tolerance for noise.  Hence, at the most fundamental level we are interested in the simultaneous approximation property of real numbers by rational numbers.  In practice, it is the  probabilistic aspect of the approximation property that is important -- knowing that  the numbers $h_1$ and $h_2$ lie within a `desirable' neighbourhood of the points $1/3$ and $2/3$ with reasonably high probability is key.  This naturally brings into  play  the  theory of metric  Diophantine approximation.

\vspace{0.5cm}

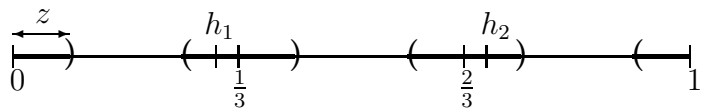
\setlength{\unitlength}{15mm}

\begin{figure}[h!]
\begin{center}
\begin{picture}(0,0)

\put(-3,0){\line(1,0){6}}
\put(-3,-0.1){\line(0,1){0.2}}
\put(3,-0.1){\line(0,1){0.2}}
\put(-1,-0.1){\line(0,1){0.2}}
\put(1,-0.1){\line(0,1){0.2}}
\put(-1.2,-0.1){\line(0,1){0.2}}
\put(1.2,-0.1){\line(0,1){0.2}}

\put(-3.025,-0.3){0}
\put(2.975,-0.3){1}
\put(-1.3,0.2){$h_1$}
\put(-1.065,-0.35){$\frac{1}{3}$}
\put(1.15,0.2){$h_2$}
\put(0.95,-0.35){$\frac{2}{3}$}

\put(-3,0.01){\line(1,0){0.5}}
\put(-3,-0.01){\line(1,0){0.5}}

\put(-1.5,0.01){\line(1,0){1}}
\put(-1.5,-0.01){\line(1,0){1}}

\put(0.5,-0.01){\line(1,0){1}}
\put(0.5,0.01){\line(1,0){1}}

\put(2.5,-0.01){\line(1,0){0.5}}
\put(2.5,0.01){\line(1,0){0.5}}

\put(-2.55,-0.05){\textbf{)}}

\put(-0.55,-0.06){\textbf{)}}
\put(-1.52,-0.06){\textbf{(}}

\put(1.45,-0.06){\textbf{)}}
\put(0.48,-0.06){\textbf{(}}

\put(2.48,-0.06){\textbf{(}}

\put(-3,0.2){\vector(1,0){0.5}}
\put(-2.5,0.2){\vector(-1,0){0.5}}

\put(-2.8,0.3){$z$}

\end{picture}
\end{center}
\end{figure}
\vspace{0.5cm}

Note that from a probabilistic point of view,  the chances that  $h_1= h_2$ is zero and is therefore insignificant. Furthermore, within the context of  this basic example,  by weighting (\emph{precoding}) the messages $u$ and $v$ appropriately before  the transmission stage  it is possible  to ensure optimal separation ($d = 1/3$) at the receiver regardless of the values of $h_1$ and $h_2$. Indeed, suppose $x_1= \frac13 h_1^{-1} u $ and  $     y_2 = \frac23 h_2^{-1} v $ are transmitted  instead of $u $ and $v$.  Then,  without taking noise into consideration,  $R$ receives the message
\begin{equation} \label{codingeg1}
y =   h_1 x_1 +  h_2 x_2  =   \textstyle{\frac13}  u + \textstyle{ \frac23}  v
\end{equation}
and so the specifics outcomes are
\begin{equation}\label{codingeg1out}
y=\begin{cases}
0 &\text{ if  } \quad  u=v=0 \\[1ex]
1/3 &\text{ if  }\quad  u=0 \text{ and } v=1\\[1ex]
2/3  &\text{ if  } \quad u=1 \text{ and } v=0\\[1ex]
1  &\text{ if  }  \quad u=v=1 \, .
\end{cases}
\end{equation}

\medskip

\noindent {\small EXAMPLE} 2. \,  There are two users  $S_1$ and $S_2$ as before but this time there are also two receivers $R_1$ and $R_2$.   Suppose $S_1$  wishes to simultaneously transmit independent signals $u_1$  and $v_1$ as a single signal, say $x_1=u_1 + v_1$ where $u_1$ is intended for  $R_1$  and $v_1$  for $R_2$. Similarly, suppose $S_2$  wishes to simultaneously transmit independent signals $u_2$  and $v_2$  as a single signal, say  $x_2=u_2 + v_2$  where $u_2$ is intended for  $R_1$  and $v_2$  for $R_2$. As in the first example, for the sake of simplicity, we can assume that the signals  $u_1, u_2, v_1, v_2 \in \{ 0,1\} $.   Now let $h_{11}$ and  $h_{21}$ denote the channel coefficients  associated with  signals  being sent by  $S_1$  to $R_1$ and $ R_2 $ respectively. Similarly, let  $h_{12}$ and  $h_{22}$ denote the channel coefficients  associated with signals being sent by  $S_2$  to $R_1$ and $R_2$.  Assume that the channel is additive and let $y_1$ (respectively $y_2$) denote the signal at receiver  $R_1 $ (respectively $R_2$).   Thus,
\begin{eqnarray} \label{2nonoise}
y_1  & = & h_{11} x_1   +  h_{12} x_2  \\ [1ex]
y_2  & = & h_{21} x_1   +  h_{22} x_2  \label{21nonoise} \,
\end{eqnarray}
where
\begin{equation} \label{2nonoiseChoice1}
 x_1= u_1 + v_1     \quad {\rm and } \quad x_2= u_2+ v_2  \, .
\end{equation}

\paragraph{}
\setlength{\unitlength}{20mm}

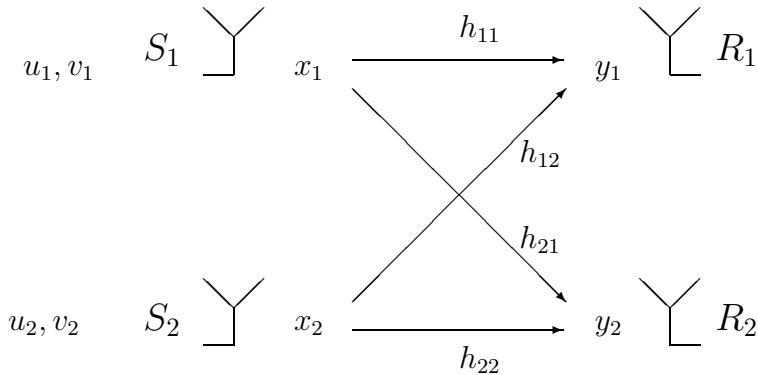
\vspace{0.5cm}

\begin{figure}[h!]
\begin{center}
\begin{picture}(2, 2)(-0.75,-0.75)

\put(-1.5,0.8){\line(0,1){0.25}}
\put(-1.5,1.05){\line(-1,1){0.2}}
\put(-1.5,1.05){\line(1,1){0.2}}
\put(-1.5,0.8){\line(-1,0){0.2}}
\put(-1.1,0.8){$x_1$}

\put(1.4,0.8){\line(0,1){0.25}}
\put(1.4,1.05){\line(-1,1){0.2}}
\put(1.4,1.05){\line(1,1){0.2}}
\put(1.4,0.8){\line(1,0){0.2}}
\put(0.9,0.8){$y_1$}

\put(-1.5,-1){\line(0,1){0.25}}
\put(-1.5,-0.75){\line(-1,1){0.2}}
\put(-1.5,-0.75){\line(1,1){0.2}}
\put(-1.5,-1){\line(-1,0){0.2}}
\put(-1.1,-0.9){$x_2$}

\put(1.4,-1){\line(0,1){0.25}}
\put(1.4,-0.75){\line(-1,1){0.2}}
\put(1.4,-0.75){\line(1,1){0.2}}
\put(1.4,-1){\line(1,0){0.2}}
\put(0.9,-0.9){$y_2$}

\put(0,0){\vector(1,1){0.71}}
\put(0,0){\line(-1,-1){0.71}}
\put(0,0){\line(-1,1){0.71}}
\put(0,0){\vector(1,-1){0.71}}

\put(-0.71,0.9){\vector(1,-0){1.4}}
\put(-0.71,-0.9){\vector(1,-0){1.4}}

\put(-2.1,0.9){\Large{$S_1$}}
\put(-2.1,-0.9){\Large{$S_2$}}
\put(1.7,0.9){\Large{$R_1$}}
\put(1.7,-0.9){\Large{$R_2$}}

\put(0,1.05){\large{$h_{11}$}}
\put(0,-1.15){\large{$h_{22}$}}
\put(0.4,0.22){\large{$h_{12}$}}
\put(0.4,-0.35){\large{$h_{21}$}}

\put(-2.9,0.8){{$u_1, v_1$}}
\put(-3.0,-0.9){{$u_2, v_2$}}

\end{picture}
\end{center}
\end{figure}

\paragraph{}
Recall, that $R_1$  (respectively $R_2$) only cares about recovering the signals $ u_1 $  and  $u_2 $ (respectively $ v_1 $  and  $v_2 $) from $y_1$ (respectively $y_2$).   For the moment, let us just concentrate on the signal  received by $R_1$; namely
\begin{equation*} \label{2nonoisee}
y_1  = h_{11} u_1 +  h_{12} u_2   + h_{11} v_1  +  h_{12} v_2 \, .
\end{equation*}
It is easily seen that this corresponds to a  received signal in Example 1 modified to incorporate four users and one receiver. This time there are potentially 16 different outcomes.  In short, the  more users, the  more outcomes and therefore the smaller the mutual  separation between them  and in turn the smaller the tolerance for noise.   Now there is one aspect of the setup in this example  that we have not yet exploited.  The receiver $R_1$  is not interested in the signals $ v_1 $  and  $v_2 $. So if they could be deliberately aligned via precoding into a single component  $v_1 + v_2 $,  then $y_1$ would look like a received signal associated with just 3 users rather than 4. With this in mind, suppose instead of transmitting $x_1$ and $x_2$ given by \eqref{2nonoiseChoice1}, $S_1$ and $S_2$ transmit the signals
\begin{equation} \label{2nonoiseChoice2}
 x_1= h_{22} u_1 + h_{12}v_1     \quad {\rm and } \quad x_2= h_{21} u_2+ h_{11} v_2  \,
\end{equation}
respectively.  Then, it can be verified that the received signals given by \eqref{2nonoise} and  \eqref{21nonoise} can be written as
\begin{eqnarray*} \label{2nonoise22}
y_1  & = & (h_{11} h_{22}) u_1   +  (h_{21} h_{12}) u_2  + (h_{11} h_{12})  (v_1 + v_2)  \\ [1ex]
y_2  & = & (h_{21} h_{12}) v_1   +  (h_{11} h_{22}) v_2  + (h_{21} h_{22})  (u_1 + u_2)  \, .
\end{eqnarray*}

\noindent In other words, the unwanted, interfering signals at either receiver are aligned  to a one dimensional subspace of four dimensional space. Notice that in the above equations the six coefficients are only of four variables, namely $h_{i,j}$, $i,j=1,2$, and thus represent dependent quantities. This, together with our findings from Example 1, naturally brings into  play  the manifold  theory of metric  Diophantine approximation.

\medskip

Example 2  is a simplified version of Example 3 appearing in \cite[\S III]{MOGMAK}.  For a deeper and more practical understanding of the link between interference alignment and metric  Diophantine approximation on manifolds the reader is urged to look at \cite{MOGMAK} and \cite[\S4.7]{J2010}.

\section{Preliminaries for Theorem~\ref{soft}} \label{section_preliminaries}

\subsection{Localisation and parameterisation}\label{loc}

Since $\cM$ is non--degenerate everywhere, we can restrict ourself to considering a sufficiently small neighbourhood of an arbitrary point on $\cM$. By compactness, $\cM$ then can be covered with a finite subcollection of such neighbourhoods. Therefore, in view of the finiteness of the cover, the existence of $\kappa_0$, $C_0$ and $C_1$ satisfying Theorem~\ref{soft} globally will follow from the existence of these parameters for every neighbourhood in the finite cover~: $\kappa_0$, $C_0$ and $C_1$ should be taken to be the minimum of their local values.

Now as we can work with $\cM$ locally, we can parameterize it with some map $\vv f:\cU\to\R^n$ defined on a ball $\cU$ in $\R^d$, where $d=\dim\cM$. Note that $\vv f$ must be at least $C^{2}$ in order to ensure that $\cM$ is non--degenerate. Without loss of generality we assume that
$$
\cM=\{\vv f(\vv x):\vv x\in \cU\}\,.
$$
Furthermore, using the Implicit Function Theorem if necessary, we can make $\vv f$ to be a Monge parametrisation, that is $\vv f(\vv x)=(x_1,\dots,x_d,f_{d+1}(\vv x),\dots,f_n(\vv x))$, where $\vv x=(x_1,\dots,x_d)$.
Note that $\vv f$ can be assumed to be bi--Lipschitz on $\cU$. This readily follows from the fact that $\vv f$ is $C^{1}$ but possibly requires a further shrinking of $\cU$.

Let $\mathcal{B}_n(\Psi,\kappa,\cM)$ denote the orthogonal projection of $\cB_n(\Psi,\kappa)\cap\cM$ onto the set of parameters $\cU$. Thus,
\begin{equation} \label{def_B_nanifold}
 \mathcal{B}_n(\Psi,\kappa,\cM):=\left\{\vv x \in \cU : \|\vv a. \vv f(\vv x) \|  > \kappa\Psi(\vv a)  \text{ ~for all~ } \vv a \in \Z^n,\ \vv a\neq\vv 0 \right\}.
\end{equation}
The set $\cB_n(\Psi,\kappa)\cap\cM$ and its projection $\mathcal{B}_n(\Psi,\kappa,\cM)$ are related by the bi--Lipschitz map $\vv f$. Since bi--Lipschitz maps only affect the Lebesgue measure of a set by a multiplicative constant (in this case the constant will depend on $\vv f$ only), it suffices to prove Theorem~\ref{soft} for the project set. More precisely, Theorem~\ref{soft} is equivalent to showing that there exist positive constants $\kappa_0$, $C_0$ and $C_1>0$ depending on $\Psi$ and $\vv f$ only such that for any $0<\delta<1$,
\begin{equation} \label{theorem_linear_forms_ie_statement_v2+}
|\cB_n(\Psi,\kappa,\cM)|_d\ge (1-\delta)|\cU|_d
\end{equation}
holds with $\kappa$ given by \eqref{kappavb}.

\subsection{Auxiliary statements}

We will denote the standard $L_1$ (resp. Euclidean, in\-fi\-ni\-ty) norm on $\R^d$ by $\left\|\,. \, \right\|_1$ (resp. $\left\|\, . \, \right\|_2$, $\left\|\,. \,  \right\|_{\infty}$). Also as before, given an $x\in\R$, $\|x\|$ will denote the distance of $x$ from the nearest integer. The notation $B\!\left(\bm{x}, r\right)$ will refer to the Euclidean open ball of radius $r>0$ centered at $\bm{x}$ and $\Sph^{d-1}$ will denote the unit sphere in dimension $d\ge 1$ (with respect to the Euclidean norm). Furthermore, throughout
$$
V_d    := \frac{\pi^{d/2}}{\Gamma \left(1+d/2\right)}
$$
is the volume of the $d$--dimensional unit ball and $N_d$ denotes the Besicovitch covering constant.

\begin{remark}\label{besicocst}
For further details on the Besicovitch covering constant, cf.~\cite{furedi_loeb}. We will only need in what follows the inequality $N_d\le 5^d$ satisfied by this constant.
\end{remark}

The proof of Theorem~\ref{soft} involves two separate cases that take into consideration the relative size of the gradient of $\vv f(\vv x)\cdot\vv q$, where $\vv q=(q_1,\dots,q_n)\in\Z^n\setminus\{\vv0\}$ and   $\vv f(\vv x)\cdot\vv q=f_1(\vv x)q_1+\dots+f_n(\vv x)q_n$ is the standard inner product of $\vv f(\vv x)$ and $\vv q$. The first case of `big gradient' is considered within the next result and is an adaptation of \cite[Theorem~1.3]{Bernik-Kleinbock-Margulis-01:MR1829381}.

In what follows, $\partial_\beta$ will denote partial derivation with respect to a multi--index $\beta=(\beta_1,\dots,\beta_d)\in\N_0^d$, where $\N_0$ will stand for the set of non--negative integers, that is $\N_0:=\{0,1,2,\dots\}$. Furthermore, $|\beta|$ we will mean the order of derivation, that is $|\beta|:=\beta_1+\dots+\beta_d$. Also, $\partial_i^k$ will denote the differential operator corresponding to the $k^{\textrm{th}}$ derivative with respect to the $i^{\textrm{th}}$ variable, that is, $\partial_i^k :=\partial^k/\partial x_i^k$.

\begin{theorem}\label{theo_explicit_1_3}
Let $\cU\subset\R^d$ be a ball of radius $r$ and $\vv f\in C^2(\wcU)$, where $\wcU$ is the ball with the same centre as $\cU$ and radius $2r$. Let
\begin{equation} \label{theo_explicit_1_3_def_L}
L^*:=\sup_{|\beta|=2,\ \vv x\in 2\cU}\|\partial_{\beta}\vv f(\vv x)\|_{\infty}\qquad\text{and}\qquad
L:=\max\left(L^*,\frac{1}{4r^2}\right).
\end{equation}
Then, for every $\delta'>0$ and every $\vv q\in\Z^n\setminus\{0\}$, the set of $\vv x\in \cU$ such that
 \[
 \|\vv f(\vv x)\cdot\vv q\|<\delta'
 \]
and
 \begin{equation} \label{theo_explicit_1_3_gradient_big}
 \|\nabla\vv f(\vv x)\vv q\|_\infty\ge\sqrt{ndL\|\vv q\|_\infty}
 \end{equation}
 has measure at most $K_d\delta'|\cU|_d$, where $\nabla\vv f(\vv x)\vv q$ is the gradient of $\vv f(\vv x)\cdot\vv q$ and
\begin{equation}\label{def_K_d}
 K_d:=\frac{4^{2d+1}d^{d/2}N_d}{V_d}
\end{equation}
is a constant depending on $d$ only.
\end{theorem}

\begin{proof}

The proof of Theorem~\ref{theo_explicit_1_3} follows on appropriately applying ~\cite[Lemma~2.2]{Bernik-Kleinbock-Margulis-01:MR1829381}. For convenience we refer to this lemma as L2.2. We take  $M$ in L2.2 to be equal to the quantity $ndL$, where $L$ is defined by~\eqref{theo_explicit_1_3_def_L}. We set $\delta$ in L2.2 to be equal to  $\delta'$ appearing in   Theorem~\ref{theo_explicit_1_3}. Then, in view of $\eqref{theo_explicit_1_3_def_L}$ and the fact that  $n,d,\|\vv q\|_\infty\geq 1$, it follows that the  hypotheses of L2.2  (namely ~\cite[Eq(2.1a) \& Eq(2.1b)]{Bernik-Kleinbock-Margulis-01:MR1829381}) are satisfied.  Thus, the conclusion of L2.2 implies Theorem~\ref{theo_explicit_1_3} with the constant $C_d$ in L2.2 equal to
$K_d$ appearing in Theorem~\ref{theo_explicit_1_3}. The explicit value of $K_d$ is calculated by `tracking' the values of the auxiliary constants $C'_d$, $C''_d$ and $C_d'''$ appearing in \cite{Bernik-Kleinbock-Margulis-01:MR1829381}.
Namely\footnote{There are two typos in the proof of L2.2 that one should be aware of when verifying the values of the constants given here. On  page~6 line~-2,  the inclusion regarding $U(x)$ is the wrong way round, it should read   $U(x)\supset B(x,\frac{\rho}{4\sqrt{d}})$.  Next, on  page~7 line 11, in the rightmost term  of the displayed set of inequalities the quantity $\delta$ is missing, it should read   $C_d'''\delta |U(x)|_d$. These typos do not affect the validity of the  proof  given in ~\cite{Bernik-Kleinbock-Margulis-01:MR1829381}.},
$$
C_d'=\frac{V_d}{2^{2d}d^{d/2}},\qquad
C_d'' = 2^{d+2},\qquad
C_d'''=\frac{C_d''}{C_d'}\,,
$$
and then
$$
K_d=2^dC_d'''N_d=\frac{2^dC_d''N_d}{C_d'} = \frac{2^{4d+2}d^{d/2}N_d}{V_d}\cdotp
$$
\end{proof}

\

Next in Theorem~\ref{theo_explicit_1_4} below we consider the case of `small gradient'.
This is an explicit version of~\cite[Theorem~1.4]{Bernik-Kleinbock-Margulis-01:MR1829381}. First we introduce auxiliary constants.

Given a $C^l$ map $\vv f:\cU\to\R^n$ defined on a ball $\cU$ in $\R^d$, the supremum of $s\in\R$ such that for any $\bm{x}\in \cU$ and any $\bm{v}\in\Sph^{n-1}$ there exists an integer $k$, $0<k\le l$, and a unit vector $\bm{u}\in\Sph^{d-1}$ satisfying
\begin{equation} \label{def_s_lb}
\left|\frac{\partial^k (\vv f\cdot \vv v)}{\partial \bm{u}^k}(\bm{x})\right|\ge s
\end{equation}
will be called \emph{the measure of  $l$--non--degeneracy of $(\bm{f},\cU)$} and will be denoted by $s(l;\vv f,\cU)$. Here and elsewhere for a unit vector $\vv u\in\Sph^{d-1}$, $\partial^k/\partial\vv u^k  $ will denote the derivative in direction $\vv u$ of order $k$.

As in Theorem~\ref{theo_explicit_1_3}, the radius of the ball $\cU$ will be denoted by $r$. Throughout, we let $\bm{x}_0$  denote the centre of $\cU$.  Also, given a real number $\lambda > 0$, we let $\lambda \cU$ denote the scaled ball of radius $\lambda \, r$ and with the same centre $\bm{x}_0$ as that of $\cU$.  With this in mind, consider the balls
\begin{equation} \label{def_tcU}
\begin{aligned}
\cUp&:=3^{d+1}\cU,\\
\tcU \, \, ~ &:=3^{n+1}\cU,\\
\tcUp&:=3^{n+d+2}\cU.
\end{aligned}
\end{equation}
For technical reasons, that will soon become apparent, in order to deal with the `small gradient' case we make the following assumption on the map $\vv f:\cU\to\R^n$.

\noindent {\bf Assumption 1. }
{ \em The map  $\vv f=(f_1,\dots,f_n)$ is an $n$--tuple of $C^{l+1}$ functions defined on the closure of $\tcUp$ which is $l$--non--degenerate everywhere on the closure of $\tcUp$.}

\medskip

 \noindent\textit{Remark.} In view of the discussion of \S\ref{loc},  there is no loss of generality in imposing Assumption~1  within the context of Theorem~\ref{soft}.

\medskip

 We denote by $s_0$ the measure of non--degeneracy of $\vv f$ on $\tcUp$. Note that
Assumption~1 ensures that
\begin{equation}\label{def_s0}
s_0:=s(l;\vv f,\tcUp)  > 0.
\end{equation}
Also, notice that it ensures the existence of a constant $M\ge 1$ such that for all $k\le l+1$ and all
$\vv u_1, \dots , \vv u_k\in\Sph^{d-1}$,
\begin{equation}
  \label{eqn:def_N}
\underset{\vv x\in \tcUp}{\sup} \left\| \frac{\partial^k \vv f}{\partial \vv u_1 \dots \partial\vv u_k}(\vv x) \right\|_{2} \le M\,,
\end{equation}
where $\partial\vv u_i$ means differentiation in direction $\vv u_i$.  Note that the left--hand side of \eqref{def_s_lb} is the length of the projection of $\partial^k\vv f(\vv x)/\partial\vv u^k$ on the line passing through $\vv v$ and hence it is no bigger than $M$. This implies that
$$
s_0  \ \leq  \  M \, .
$$

Without loss of generality,   we will assume that the radius $r$ of the ball $\cU$  satisfies
\begin{equation}\label{slv2}
  r\le \min\left\{\frac{s_0 \cdot \sigma(l,d)}{2\cdot 3^{n+d+2}\sqrt{d}M}, \frac{\eta s_0}{4\cdot10^73^{n+d+2}\, d M l^{l+2}(l+1)!}\right\},
\end{equation}
where
\begin{equation}\label{def_r_V}
\eta:= \min\left\{\frac{1}{16} , \left( \frac{V_d}{2^{d+2}dl(l+1)^{1/l}5^d}\right)^{d(2l-1)(2l-2)}\right\}
\end{equation}
and where
\begin{equation}\label{def_sigma_l_d}
\sigma(l,d):= \frac{1}{2^{3l(d-1)/2}}\cdot \phi\left(\left( \sqrt{2}\cdot(2l)^{2+l(l-1)/2}\cdot(l+1)!\right)^{-1}, 2, l \right)^{d-1}
\end{equation}
with the quantity $\phi(\omega, B, k)$ defined as
\begin{equation}\label{def_phi_deltaBk}
\phi(\omega, B, k):= \frac{\omega^{k(k-1)/2}}{2\sqrt{2}\cdot B^k\cdot (k+1)!}
\end{equation}
for any integer $k\ge 1$ and any real numbers $\omega, B>0$.

Furthermore, define the following constants determined by $\vv f$ and $\cU$:
\begin{equation} \label{def_rho_one_first}
\rho_1:=\frac{s_0}{4l^l(l+1)!\sqrt{d}}(2r)^l\,,
\end{equation}
$$
\tau:= \frac{r^l s_0}{4l^l (l+1)!},
$$
and
\begin{equation} \label{def_rho_two_first}
\rho_2:=\frac{s_0}{4l^l(l+1)!}\left( \frac{\tau}{M}\right)^{l-1}\frac{\left(\tau\left( 1-1/\sqrt{2}\right)\right)^2}{\sqrt{\left(\frac{s_0}{4l^l (l+1)!}\left(\frac{\tau}{M}\right)^l \right)^2+\left(\tau\left(1+\frac{1}{\sqrt{2}}\right)\right)^2}}\,\cdotp
\end{equation}
Finally, let
\begin{equation}\label{deffinalerho_first}
\rho:=\frac{\rho_1\rho_2}{\sqrt{\rho_1^2+(\rho_2+2 M^2)^2}}\cdotp
\end{equation}

\medskip

\begin{theorem}\label{theo_explicit_1_4}
Let $\cU\subset\R^d$ be a ball and  $\vv f=(f_1,\dots,f_n)$ be an $n$--tuple of $C^{l+1}$ functions satisfying  Assumption 1. 
Then, for any $0<\delta'\le{}1$, any $n$-tuple $\vv T=(T_1, \dots, T_n)$ of real numbers $\ge{}1$ and any $K>0$ satisfying
\begin{equation} \label{theo_explicit_1_4_restrictions}
\frac{\delta' K T_1\cdot\dots\cdot T_n}{\max_{i=1,\dots,n}T_i}\le 1,
\end{equation}
define the set $A(\delta',K,\vv T)$ to be
\begin{equation} \label{theo_explicit_1_4_set}
A(\delta',K, \vv T):=\left\{
\vv x\in \cU : \exists \ \vv q\in\Z^n\setminus\{0\} \text{ such that }
\begin{cases}
	&\|\vv f(\vv x)\cdot\vv q\|<\delta'\\
  &\|\nabla\vv f(\vv x){}\vv q\|_\infty<K\\
  &|q_i|<T_i,i=1,\dots,n
\end{cases}
\right\}.
\end{equation}
Then
\[ 
\left| A(\delta', K, \vv T)\right|_d \le E\left(\sqrt{n+d+1}\cdot\varepsilon_1\right)^{1/d(2l-1)}|\cU|_d \, ,
\]
where
\begin{equation} \label{theo_explicit_1_4_def_e1}
\varepsilon_1:=\max\left(\delta',\left(\frac{\delta' K T_1\cdot\dots\cdot T_n}{\max\limits_{1\le i \le n}T_i}\right)^{\frac{1}{n+1}}\right)\,
\end{equation}
and
\begin{equation} \label{theo_explicit_1_4_def_E}
E:=C(n+1)(3^dN_d)^{n+1}\rho^{-1/d(2l-1)}\,,
\end{equation}

\noindent in which  $\rho$ is given by~\eqref{deffinalerho_first} and $C$ is the constant explicitly given by ~\eqref{def_C} below.
\end{theorem}

At first glance the statement of Theorem~\ref{theo_explicit_1_4} looks very similar to \cite[Theorem~1.4]{Bernik-Kleinbock-Margulis-01:MR1829381}. We stress that the key difference is that in our statement the constants are made fully explicit.
The  proof of Theorem~\ref{theo_explicit_1_4} is rather involved and will be the subject of \S\ref{xmas}.

\section{A strengthening and proof of Theorem~\ref{soft} \label{6}}

In view of the discussion of \S\ref{loc}, Theorem~\ref{soft} will follow immediately on establishing a stronger result
(Theorem~\ref{theorem_expicit_1_1} below), which explicitly characterizes the dependence on $\Psi$ and $\cM$ of the constants $\kappa_0$, $C_0$ and $C_1$  appearing within the statement of Theorem~\ref{soft}. In the case that the function $\vv f$ defining the manifold under consideration is explicitly given, the values of these constants may be improved by following the methodology of the proof of Theorem~\ref{theorem_expicit_1_1} as many computations will then be made simpler.

Let
\begin{equation}\label{slv3}
C_{\Psi}:=\sup_{\vv q=(q_1,\dots,q_n)\in\Z^n\setminus\{\vv0\}}\Psi(\vv q)\prod_{i=1}^nq_i^+,\qquad\text{where }q_i^+:=\max\{1,|q_i|\}\,.
\end{equation}
It is a well known fact that, under the assumption that $\Psi$ is monotonically decreasing in each variable, relation~\eqref{def_S_Psi} implies that $0<C_\Psi<\infty$.
Also define the constant
\[
S_n:=\sum_{\vv t\in \Z^n}2^{-\frac{\|\vv t\|_{\infty}}{2d(2l-1)(n+1)}},
\]
which is clearly finite and positive as the sum converges.

\begin{theorem}\label{theorem_expicit_1_1}
Let $\cU\subset\R^d$ be a ball whose radius satisfies~\eqref{slv2} and let  $\vv f=(f_1,\dots,f_n)$ be an $n$--tuple of $C^{l+1}$ functions satisfying  Assumption 1.
Let $\Sigma_{\Psi}$, $L$, $K_d$ and $E$ be given by\eqref{def_S_Psi}, \eqref{theo_explicit_1_3_def_L}, \eqref{def_K_d} and~\eqref{theo_explicit_1_4_def_E} respectively and let
$$
K_0=E\left(\sqrt{n+d+1}\cdot
\left(C_{\Psi}2^{n-1/2}\sqrt{ndL}\right)^{\frac{1}{n+1}}\right)^{1/d(2l-1)}\,.
$$

Given any $\delta>0$, let
\[
\kappa:=\min\left\{\frac{1}{C_{\Psi}2^{n-1/2}\sqrt{ndL}},\ \frac{\delta}{2K_d  \Sigma_{\Psi}},\left(\frac{\delta}{2K_0S_n}\right)^{d(n+1)(2l-1)}\right\}.
\]
Then
\begin{equation}  \label{theorem_expicit_1_1_conclusion_Lebesgue++}
\left|\mathcal{B}_n(\Psi,\kappa,\cM)\right|_d\ge (1-\delta)|\cU|_d.
\end{equation}
\end{theorem}

Clearly the above is an explicit version of Theorem~\ref{soft} in the case when $\mu$ is Lebesgue measure. The arguments given in the proof of Theorem~\ref{theorem_linear_forms+} are easily adapted to deal with the general situation.

\subsection{Proof of Theorem~\ref{theorem_expicit_1_1} modulo Theorem~\ref{theo_explicit_1_4}}

For $\kappa>0$ and any $\vv q\in\Z^n$, define

$$
A(\kappa;\vv q):=\left\{ \vv x\in \cU :  \|\vv f(\vv x)\cdot\vv q\|<\kappa \Psi(\vv q) ~ \& ~\eqref{theo_explicit_1_3_gradient_big}\text{ holds}\right\}
$$
and
$$
A^c(\kappa;\vv q):=\left\{ \vv x\in \cU :  \|\vv f(\vv x)\cdot\vv q\|<\kappa \Psi(\vv q) ~ \& ~\eqref{theo_explicit_1_3_gradient_big}\text{ does not hold}\right\}\,.
$$
Clearly it suffices to prove that
\begin{equation}  \label{theorem_expicit_1_1_conclusion_Lebesgue}
\left|\bigcup_{\vv q\in\Z^n\setminus\{0\}}A(\kappa;\vv q)\right|_d\le \frac\delta2|\cU|_d\qquad\text{and}\qquad
\left|\bigcup_{\vv q\in\Z^n\setminus\{0\}}A^c(\kappa;\vv q)\right|_d\le \frac\delta2|\cU|_d\,.
\end{equation}

By Theorem~\ref{theo_explicit_1_3} with $\delta'=\kappa\Psi(\vv q)$, we immediately have that $|A(\kappa;\vv q)|_d\le K_d\kappa\Psi(\vv q)|\cU|_d$. Then, summing over all $\vv q\in\Z^n\setminus\{0\}$ gives
\begin{equation}  \label{theorem_expicit_1_1_first_case_result}
\left|\bigcup_{\vv q\in\Z^n\setminus\{0\}}A(\kappa;\vv q)\right|_d\le K_d\Sigma_\Psi\kappa|\cU|_d\le\frac{\delta}{2}\,|\cU|_d.
\end{equation}
Now to establish the second inequality in \eqref{theorem_expicit_1_1_conclusion_Lebesgue}, given an $n$--tuple $\vv t=(t_1,\dots,t_n)\in\N_0^n$, define the set
\begin{equation} \label{theorem_expicit_1_1_union}
A^c_{\vv t}:=\bigcup_{\substack{\vv q=(q_1,\dots,q_n)\in \Z^n\setminus\{\vv0\}\\[1ex] 2^{t_i}\le q_i^+< 2^{t_i+1}}}  \!\!\! A^c(\kappa;\vv q)\,,
\end{equation}
where $q_i^+=\max\{1,|q_i|\}$. Observe that
\begin{equation} \label{theorem_expicit_1_1_union_two}
\bigcup_{\vv q\in \Z^n}A^c(\kappa;\vv q)=\bigcup_{\vv t\in\N_0^n}A^c_{\vv t}\,.
\end{equation}
By \eqref{slv3} and the monotonicity of $\Psi$ in each variable, for every $\vv q=(q_1,\dots,q_n)\in \Z^n\setminus\{\vv0\}$ satisfying the inequalities $2^{t_i}\le q_i^+< 2^{t_i+1}$, we have that
$$
\Psi(\vv q)\le C_\Psi \left(\prod_{i=1}^nq_i^+\right)^{-1}\le C_\Psi \prod_{i=1}^n2^{-t_i}=C_\Psi 2^{-\sum_{i=1}^nt_i}
$$
and
$$
\|\vv q\|_\infty\le 2^{\max_it_i+1}\,.
$$
Now let
$$
\delta'=\kappa C_{\Psi}2^{-\sum_{i=1}^n t_i},\quad K=\sqrt{ndL 2^{\max_it_i+1}}\quad\text{and}\quad T_i=2^{t_i+1}\ (1\le i\le n)\,.
$$
Then, $A^c_{\vv t}$ is easily seen to be contained in the set $A(\delta', K, \vv T)$ defined within Theorem~\ref{theo_explicit_1_4}.
Clearly $T_1, \dots, T_n\ge 1$ and $K>0$. Since $\kappa<C_{\Psi}^{-1}$, we have that $0<\delta'<1$. Finally, \eqref{theo_explicit_1_4_restrictions} is satisfied, since
\begin{align}
\nonumber \frac{\delta' K T_1\cdot\dots\cdot T_n}{\max_{i=1,\dots,n}T_i}
&=\frac{\kappa C_{\Psi}2^{-\sum_{i=1}^n t_i}\sqrt{ndL 2^{\max_it_i+1}} \prod_i2^{t_i+1}}{2^{\max_i t_i+1}}\\[1ex]
\label{slv4}&=\frac{\kappa C_{\Psi}2^{n}\sqrt{ndL}}{2^{(\max_i t_i+1)/2}}
 \; =  \; \frac{\kappa C_{\Psi}2^{n-1/2}\sqrt{ndL}}{2^{\|\vv t\|_\infty/2}}  \; <  \;1\,,
\end{align}
where the last inequality follows from the definition of $\kappa$. Therefore,  Theorem~\ref{theo_explicit_1_4} is applicable and it follows that
$$
|A^c_{\vv t}|_d\le |A(\delta', K, \vv T)|_d\le
E\left(\sqrt{n+d+1}\cdot\varepsilon_1\right)^{1/d(2l-1)}|\cU|_d\,,
$$
where $E$ is given by \eqref{theo_explicit_1_4_def_E} and where, from~\eqref{slv4}, the definition of $\delta'$ and the fact that $\kappa C_{\Psi}<1$,
$$
\varepsilon_1=\max\left(\kappa C_{\Psi}2^{-\sum_{i=1}^n t_i},\left(\frac{\kappa C_{\Psi}2^{n-1}\sqrt{ndL}}{2^{\|\vv t\|_\infty/2}}\right)^{\frac{1}{n+1}}\right)=
\left(\frac{\kappa C_{\Psi}2^{n-1}\sqrt{ndL}}{2^{\|\vv t\|_\infty/2}}\right)^{\frac{1}{n+1}}.
$$
Then, using \eqref{theorem_expicit_1_1_union_two} and summing over all $\vv t\in\N_0^n$, we find that
\begin{eqnarray*}
\left|\bigcup_{\vv q\in \Z^n}A^c(\kappa;\vv q)\right|_d  & \le  &
\sum_{\vv t\in\N_0^n}E\left(\sqrt{n+d+1}\cdot
\left(\frac{\kappa C_{\Psi}2^{n-1}\sqrt{ndL}}{2^{\|\vv t\|_\infty/2}}\right)^{\frac{1}{n+1}}\right)^{1/d(2l-1)}\!\!\! |\cU|_d  \\[1ex]
& =  &  K_0S_n\kappa^{1/d(n+1)(2l-1)} |\cU|_d   \ \le  \ \frac\delta2|\cU|_d\, ,
\end{eqnarray*}
where the latter inequality follows from the definition of $\kappa$. This establishes the secound inequality in \eqref{theorem_expicit_1_1_conclusion_Lebesgue} and thus completes the proof of Theorem~\ref{theorem_expicit_1_1} modulo Theorem \ref{theo_explicit_1_4}.

\section{Proof of Theorem~\ref{theo_explicit_1_4} \label{xmas}}
\label{section_quantitative_BKM}

To establish Theorem~\ref{theo_explicit_1_4}, we will follow the basic strategy set out in the proof of~\cite[Theorem 1.4]{Bernik-Kleinbock-Margulis-01:MR1829381}. We stress that non--trivial modifications and additions are required to make the constants explicit.
To begin with, we state a simplified form of~\cite[Theorem~6.2]{Bernik-Kleinbock-Margulis-01:MR1829381} and, to this end, various notions are now introduced.

Given a finite dimensional real vector space $W$, $\nu$ will denote a submultiplicative function on the exterior algebra $\bigwedge W$;  that is, $\nu$ is  a continuous function from $\bigwedge W$ to $\Rp$ such that
\begin{equation}\label{newnew}
\nu(t\vv w)=|t|  \, \nu(\vv w)   \quad \rm{ and} \quad \nu(\vv u\wedge\vv w)\le \nu(\vv u)\nu(\vv w)
\end{equation}
for any $t\in\R$ and  $\vv u,\vv w\in\bigwedge W$. Given a discrete subgroup $\Lambda$ of $W$ of rank $k\ge 1$, let $\nu(\Lambda):=\nu(\vv v_1\wedge\dots\wedge\vv v_k)$, where $\vv v_1,\dots,\vv v_k$ is a basis of $\Lambda$ (this definition makes sense from the first equation in~\eqref{newnew}). Also, $\cL(\Lambda)$ will denote the set of all non--zero primitive subgroups of $\Lambda$. Furthermore, given $C,\alpha>0$ and $V\subset\R^d$, a function $f: \vv x\in V\mapsto{}f(\vv x)\in \R$ is said to be $(C,\alpha)$--good on $V$ if for any open ball $B\subset{}V$ and any $\epsilon>0$,
\[
	|\{ \vv x\in{}B: |f(\vv x)| < \epsilon \sup_{\vv x\in{}B} | f(\vv x) |\}|_d \ \le\ C\epsilon^\alpha|B|_d.
\]

\begin{theorem}(\cite[Theorem~6.2]{Bernik-Kleinbock-Margulis-01:MR1829381}) \label{theo_BKM}
Let $W$ be a $d+n+1$ dimensional real vector space and $\Lambda$ be a discrete subgroup of $W$ of rank $k$. Let a ball $B=B(\vv{x_0},r)\subset\R^d$ and a map $H:\hcB\rightarrow \mathrm{GL}(W)$ be given, where $\hcB:= 3^{k}B$. Take $C,\alpha>0$, $0<\tilde{\rho}<1$ and let $\nu$ be a submultiplicative function on $\bigwedge W$. Assume that for any $\Gamma\in\cL(\Lambda)$,
\begin{enumerate}
\item[(i)] \label{theo_BKM_condition_Ca} the function $\vv x\to\nu(H(\vv x)\Gamma)$ is $(C,\alpha)$--good on $\hcB$,\\[-1ex]
\item[(ii)] \label{theo_BKM_condition_rho} there exists $\vv x\in B$ such that $\nu(H(\vv x)\Gamma)\ge\tilde{\rho}$, and\\[-1ex]
\item[(iii)] \label{theo_BKM_condition_finite} for all $\vv x\in\hcB$, $\#\{\Gamma\in\cL(\Lambda)\mid\nu(H(\vv x)\Gamma)<\tilde{\rho}\}<\infty$.
\end{enumerate}
Then, for any positive $\varepsilon\le\rho$, one has
\begin{equation}\label{slv1}
\left|\left\{\vv x\in B:\exists~ \vv v\in\Lambda\setminus\{0\} \text{ such that }\nu(H(\vv x)\vv v)<\varepsilon\right\}\right|_d~\le~ k(3^dN_d)^kC\left(\frac{\varepsilon}{\rho}\right)^{\alpha}|B|_d.
\end{equation}
\end{theorem}

Theorem~\ref{theo_explicit_1_4} is now deduced from Theorem~\ref{theo_BKM} in the following manner. With respect to the parameters appearing in Theorem~\ref{theo_BKM}, we let  $$W= \R^{n+d+1}$$ and  $$ \nu_{*} \mbox{ be the submultiplicative function introduced in~\cite[\S7]{Bernik-Kleinbock-Margulis-01:MR1829381}}. $$

\noindent There is nothing to gain in formally recalling the definition of $\nu_{*}$.  All we need to know is that $\nu_{*}$ as given in \cite{Bernik-Kleinbock-Margulis-01:MR1829381} has the property that
\begin{equation} \label{property_nu}
\nu_{*}(\vv w)\le\|\vv w\|_2  \qquad \forall  \ \ \vv w\in \textstyle{\bigwedge W}
\end{equation}
and that its restriction to $W$ coincides with the Euclidean norm. Next, the discrete subgroup $\Lambda$ appearing in Theorem~\ref{theo_BKM} is defined as
\begin{equation} \label{def_Lambda}
\Lambda:=\left\{\begin{pmatrix}p\\\bm{0}\\\vv q\end{pmatrix}\in \R^{n+d+1}  :  p\in\Z, \vv q \in \Z^n \right\}.
\end{equation}
Note that it has rank $k=n+1$, therefore the ball $\hcB$ appearing in the statement of Theorem~\ref{theo_BKM} coincides with the ball $\tcU$  defined by~\eqref{def_tcU}. Finally, we let the map $H$ send $\vv x\in \tcU$ to the product of matrices
\begin{equation} \label{def_H}
H({\vv x}):=DU_{\vv x},
\end{equation}
where
\begin{equation} \label{def_Ux}
U_{\vv x}:=\begin{pmatrix}1&0&\vv f( \vv x)\\0&I_d&\nabla\vv f(\vv x)\\0&0&I_n\end{pmatrix}\in \mathrm{SL}_{n+d+1}(\R)
\end{equation}
and  $D$ is the diagonal matrix
\begin{equation} \label{def_D}
D:=\diag\Big(\frac{\varepsilon_1}{\delta'},\underbrace{\frac{\varepsilon_1}{K},\dots,\frac{\varepsilon_1}{K}}_{d \textrm{ times}},
\frac{\varepsilon_1}{T_1},\dots,\frac{\varepsilon_1}{T_n}\Big)
\end{equation}
 defined via  the constants $\delta'$, $K$, $T_1,\dots,T_n$, and $\varepsilon_1$ appearing in Theorem~\ref{theo_explicit_1_4}.

With the above choice of parameters,  on using~\eqref{property_nu}, it is  easily verified that the set $A(\delta',K, \vv T)$ defined by~\eqref{theo_explicit_1_4_set} within the context of Theorem~\ref{theo_explicit_1_4}  is contained in the set on the left--hand side of~\eqref{slv1} with
$\varepsilon:=\varepsilon_1\sqrt{n+d+1}$. The upshot is that Theorem~\ref{theo_explicit_1_4}  follows from Theorem~\ref{theo_BKM} on verifying conditions (i), (ii) and (iii)  therein with appropriate constants $C$, $\alpha$ and $\rho$.  With this in mind, we note that condition (iii) is already established in~\cite[\S7]{Bernik-Kleinbock-Margulis-01:MR1829381} for any $\tilde{\rho}\le1$.  In~\S\ref{cond1-3} below, we will verify the remaining conditions (i) \& (ii) with the following explicit constants~:
\begin{equation} \label{def_C}
C:=\left(\frac{d(n+2)(d(n+1)+2)}{2}\right)^{\alpha/2}\max\left(C_1^*,2C_{d,l}\right),
\end{equation}
where 
\begin{equation}
	C_1^*:=\max\left(\frac{2 M}{s_0 \cdot\sigma(l,d)} ,\,  \frac{2^{d+2}}{V_d} \cdot dl(l+1)\cdot\frac{M}{s_0}\cdot\left(\frac{2l^l+1}{\sigma(l,d)}\right)^{1/l}\right)
\label{def_C_3.5}
\end{equation}
(here, $\sigma(l,d)$ is the quantity defined in~\eqref{def_sigma_l_d}),
\begin{equation} \label{def_Cdl}
C_{d,l}:= \frac{2^{d+1}dl(l+1)^{1/l}}{V_d},
\end{equation}
\begin{equation} \label{def_alpha}
\alpha:= \frac{1}{d(2l-1)}
\end{equation}
and $\tilde{\rho} = \rho$ as defined by~\eqref{deffinalerho_first} (note that $\rho<1$). This will establish Theorem~\ref{theo_explicit_1_4}.

\section{Verifying conditions (i) \& (ii) of Theorem~\ref{theo_BKM}\label{cond1-3}}

Unless stated otherwise, throughout this section,  $\Lambda$ will be  the discrete subgroup given  by~\eqref{def_Lambda} and  $\Gamma\in\cL(\Lambda)$ will be a primitive subgroup of $\Lambda$.  Verifying condition (i)  of Theorem~\ref{theo_BKM} is based on two separate cases~: one when the rank of $\Gamma$ is one and the other case of rank $\ge 2$.

\subsection{ Rank one case of condition (i) \label{900}}

The key to verifying condition (i)  in the  case that $\Gamma$ is of rank one is  the following explicit version of~\cite[Proposition 3.4]{Bernik-Kleinbock-Margulis-01:MR1829381}.  Notice that it and its corollary are themselves independent of rank  and indeed $\Gamma$.

\begin{proposition}
\label{proposition_explicit_3_4}
Let $\cU\subset\R^d$ be a ball, $\mathcal{F}\subset \mathcal{C}^{l+1}\left( \tcUp\right)$ be a family of real valued functions and $\lambda$ and $\gamma$ be positive real numbers such that~:
\begin{itemize}
\item[(1)] the set $\left\{ \nabla{f}\, : \, f\in\mathcal{F}\right\}$ is compact in $\mathcal{C}^{l-1}\left( \tcUp\right)$,\\[0ex]
\item[(2)] $\sup\limits_{f\in\mathcal{F}}\underset{\bm{x}\in \tcUp}{\sup}\left|\partial_{\vv\beta}f(\bm{x})\right| \le \lambda$ for any multi--index $\vv\beta \in \N_0^d$ with $1\le \left|\vv\beta\right| \le l+1$,\\[0ex]
\item[(3)]
$\underset{f\in\mathcal{F}}{\inf} \;\underset{\vv u \in \Sph^{d-1}}{\sup} \; \underset{1\le k \le l}{\sup} \;\left| \frac{\partial^k f}{\partial\vv u^k}(\vv x_0) \right| \ge \gamma$, where $\vv x_0$ is the centre of $\cU$,\\[1ex]
\item[(4)]
$
\displaystyle\frac{\gamma\cdot \sigma(l,d)}{2\cdot3^{n+d+2}\sqrt{d}\lambda}\ge r\,,
$
where $r$ is the radius of $\cU$ as defined in~\eqref{slv2} and $\sigma(l,d)$ is defined in~\eqref{def_sigma_l_d}.
\end{itemize}
Then, for any $f\in\mathcal{F}$, we have that\\
\begin{itemize}
\item[(a)] ~~$f$ is $\left(C_1, \frac{1}{dl} \right)$--good on $\tcU$,
\item[]
\item[(b)] ~~$\left\|\nabla{f}\right\|_{\infty}$ is $\left(C_1, \frac{1}{d(l-1)} \right)$--good on $\tcU$,
\end{itemize}
where
\begin{equation}
	\label{C_explicit_3.4}
C_1=C_1(\gamma,\lambda):=\max\left(\frac{2\lambda}{\gamma\cdot\sigma(l,d)} ,  \frac{2^{d+2}}{V_d} \, dl(l+1)\frac{\lambda}{\gamma}\left(\frac{2l^l+1}{\sigma(l,d)}\right)^{1/l}\right).
\end{equation}
\end{proposition}

\medskip

\noindent\textit{Remark.}  Hypothesis (2) is  additional to those made in~\cite[Proposition 3.4]{Bernik-Kleinbock-Margulis-01:MR1829381}. In short, it is this ``extra'' hypothesis that yields an explicit formula for the constant~$C_1$. Note that by the  definition of $C_1^* $ as  given by~\eqref{def_C_3.5},  we have that  $$C_1^* =C_1(s_0,M) \, . $$

\medskip

Using the explicit constant $C_1$  appearing in Proposition~\ref{proposition_explicit_3_4}, it is possible to adapt the proof of \cite[Corollary 3.5]{Bernik-Kleinbock-Margulis-01:MR1829381} to give the following statement.

\begin{corollary}
\label{explicoro3.5}
Let $\cU\subset\R^d$ be a ball and  $\vv f=(f_1,\dots,f_n)$ be an $n$--tuple of $C^{l+1}$ functions satisfying  Assumption 1. 
With reference to Proposition~\ref{proposition_explicit_3_4}, let
$$
\gamma:= s_0\qquad \textrm{and }\qquad \lambda:=M.
$$
Then, for any linear combination $f=c_0+\sum_{i=1}^nc_i f_i$ with $c_0,\dots,c_n\in\R$, we have that
\begin{itemize}
\item[(a)] $f$ is $\left(C_1^*, \frac{1}{dl} \right)$--good on $\tcU$,
\item[]
\item[(b)] $\|\nabla f\|_\infty$ is $\left(C_1^*, \frac{1}{d(l-1)} \right)$--good on $\tcU$.
\end{itemize}
\end{corollary}

Corollary~\ref{explicoro3.5} allows us to verify condition~(i) of Theorem~\ref{theo_BKM} in the case that $\Gamma$ is a primitive subgroup of $\Lambda$ of rank 1. Indeed, in view of \eqref{def_H} and of the discussions following equations~\eqref{newnew} and~\eqref{property_nu}, $\nu_*(H(\vv x)\Gamma)$ is the Euclidean norm of $H(\vv x)\vv w=DU_{\vv x}\vv w$, where $\vv w$ is a basis vector of $\Gamma$. It is readily seen that the coordinate functions of $H(\vv x)\vv w$ are either constants, or $f(\vv x)$, or $\partial f(\vv x)/\partial x_i$ for some $f=c_0+\sum_{i=1}^nc_i f_i$ with $c_0,\dots,c_n\in\R$. Hence, by Corollary~\ref{explicoro3.5} and ~\cite[Lemma 3.1~(b,d)]{Bernik-Kleinbock-Margulis-01:MR1829381} we obtain that the function $\|H(\cdot)\Gamma\|_\infty$ is $\left(C_1^*, \alpha\right)$--good on $\tcU$, where $\alpha$ is given by \eqref{def_alpha}. In turn,  on using~\cite[Lemma~3.1(c)]{Bernik-Kleinbock-Margulis-01:MR1829381} and the fact that
$$
\frac{1}{\sqrt{n+d+1}}\leq\frac{\|H(\vv x)\vv w\|_{\infty}}{\|H(\vv x)\vv w\|_2}\leq 1,
$$
we have that $ \nu_{*} \left(H(\, . \,  )\Gamma\right)  \  {\rm  \ is \ }  \left(C_1^*(n+d+1)^{\alpha/2},\alpha\right)-{\rm good \ on\ }\tilde\cU.
$ It then follows from~\cite[Lemma~3.1(d)]{Bernik-Kleinbock-Margulis-01:MR1829381} that $$ \nu_{*} \left(H(\, . \,  )\Gamma\right)  \  {\rm  \ is \ }  \left(C,\alpha\right)-{\rm good \ on\ }\tilde\cU.$$

\begin{proof}[Proof of Corollary \ref{explicoro3.5}]
In view of \cite[Lemma 3.1.a]{Bernik-Kleinbock-Margulis-01:MR1829381}, it suffices to prove the corollary under the assumption that  $\left\|\left(c_1, \dots, c_n\right)\right\|_2 =1$. Thus, with reference to Proposition~\ref{proposition_explicit_3_4}, define
$$\mathcal{F}:= \left\{ c_0+\sum_{i=1}^{n}c_i f_i \, : \, \left\|\left(c_1, \dots, c_n\right)\right\|_2 =1\right\}  \, . $$ The corollary will follow on verifying  the four hypotheses of  Proposition~\ref{proposition_explicit_3_4}.  Thus, hypothesis (1) is easily seen to be satisfied. Hypothesis~(2) is a consequence of \eqref{eqn:def_N} and of the Cauchy--Schwarz inequality while  hypothesis (3) follows straightforwardly from the definition of the measure of non--degeneracy $s_0$ in~\eqref{def_s_lb} and~\eqref{def_s0}. Finally, hypothesis (4) is guaranteed by~\eqref{slv2} and the choices of  $\gamma$ and $\lambda$.
\end{proof}

\subsection{Proof of Proposition~\ref{proposition_explicit_3_4}}

The proof of Proposition~\ref{proposition_explicit_3_4} relies on the following lemma~:

\begin{lemma}\label{cornerstoneproofprop34}
Let $f$ be a real--valued function of class $C^k$ ($k\ge 1$) defined in a neighbourhood of $\vv x\in \R^d$ ($d\ge 1$). Assume that there exists an index $1\le i_0\le d$ and a real number $\mu>0$ such that $$\left|\frac{\partial^k f}{\partial x_{i_0}^k}(\vv x)\right|\ge \mu.$$
Then there exists a rotation $S : \R^d\rightarrow\R^d$ such that
$$\left|\frac{\partial^k \left(f\circ S\right)}{\partial x_i^k}(\vv x)\right|\,\ge \, \mu\cdot \sigma(k,d)$$ for all indices $1\le i\le d$, where the quantity $\sigma(k,d)$ is defined in\eqref{def_sigma_l_d}.
\end{lemma}

As the proof of Lemma~\ref{cornerstoneproofprop34} is lengthy, before given it, we show how to deduce Proposition~\ref{proposition_explicit_3_4} from it.

\noindent\textit{Deduction of Proposition~\ref{proposition_explicit_3_4} from Lemma~\ref{cornerstoneproofprop34}.} Let $\vv x_0=(v_1, v_2, \dots, v_d )$ denote the centre of $\cU$. Hypothesis (3) of Proposition~\ref{proposition_explicit_3_4} implies that for any $f\in\mathcal{F}$, there exists a
unit vector $\vv u\in\Sph^{d-1}$ and an index $1\le k \le l$ such that $$\left|\frac{\partial^k f}{\partial \vv u^k}(\vv x_0)\right|\ge \gamma.$$
Even if it means applying a first rotation to the coordinate system that brings the $x_1$ axis onto the line spanned by the vector $\vv u$, it may be assumed, without loss of generality, that the above inequality reads as $$\left| \partial_1^k f (\vv x_0)\right|\ge \gamma.$$
From Lemma~\ref{cornerstoneproofprop34}, up to another rotation of the coordinate system, one can guarantee that $$\left| \partial_i^k f (\vv x_0)\right|\ge \gamma\cdot \sigma(k,d)\,:=\, C_2$$ for all indices $1\le i \le d$.
Now, for a fixed index $i$,  it follows from a Taylor expansion at $ \vv x_0$ that, for any $\bm{x}=(x_1, \dots, x_d ) \in \tcUp$,
\[
\partial_i^k f\left(\bm{x}\right) = \partial_i^k f\left(\vv x_0\right) + \sum_{j=1}^{d} R_j(\vv x;\vv x_0) \left(x_j - v_{j}\right),
\]
where, by hypothesis (2),
$R_j(\vv x;\vv x_0)$ satisfies the inequality
\[
\left|R_j(\vv x;\vv x_0)\right| \, \le\, \underset{\bm{x}\in \tcUp}{\sup}\left|\left(\partial_{j}\circ \partial_i^k\right)f(\bm{x})\right| \le \lambda.
\]
In view of hypothesis~(4), we have furthermore that
$$
\left\| \bm{x} - \vv x_0\right\|_2\leq 3^{n+d+2}r\leq \frac{\gamma\cdot \sigma(l,d)}{2\sqrt{d}\lambda}\leq\frac{C_2}{2\sqrt{d}\lambda}\cdotp
$$
Thus, for all indices $1\le i\le d$,
\begin{align}
\left|\partial_i^k f\left(\bm{x}\right)\right| \, &\ge \, C_2 - \sum_{i=1}^{d}\left|x_j - u_j\right| \lambda \, = \, C_2-\lambda\left\|\bm{x} - \vv x_0\right\|_1 \nonumber\\
&\ge \, C_2 - \lambda \sqrt{d} \left\| \bm{x} - \vv x_0\right\|_2 \
 \ge \, \frac{C_2}{2}\cdotp\label{vb1001}
\end{align}

\noindent Next, observe that any cube circumscribed about $\tcU$ lies inside of $\tcUp$. It then follows on applying~\cite[Lemma 3.3]{Bernik-Kleinbock-Margulis-01:MR1829381} with $A_1 =\lambda$ and $A_2= C_2/2$ that the function $f$ is $\left(C', \frac{1}{dk}\right)$--good on $\tcU$, where
\begin{align*}
C' &:= \frac{2^d}{V_d}dk (k+1)\left(\frac{2\lambda}{\gamma\cdot \sigma(k,d)}(k+1)\left(2k^k+1\right) \right)^{1/k} \\[1ex]
&\le\, \frac{2^{d+2}}{V_d}dk(k+1)\frac{\lambda}{\gamma}\left(\frac{2k^k+1}{\sigma(k,d)} \right)^{1/k}.
\end{align*}
A computation then shows that
\begin{align*}
C' \le\, \frac{2^{d+2}}{V_d}dl(l+1)\frac{\lambda}{\gamma}\left(\frac{2l^l+1}{\sigma(l,d)} \right)^{1/l}.
\end{align*}
Part (a) of Proposition~\ref{proposition_explicit_3_4} is now a consequence of~\cite[Lemma 3.1.d]{Bernik-Kleinbock-Margulis-01:MR1829381}. Regarding  part (b), the proof is  essentially the same as that of~\cite[Proposition~3.4.b]{Bernik-Kleinbock-Margulis-01:MR1829381} with the constant $C$ replaced with the explicit constant $C_1$ given by~\eqref{C_explicit_3.4}.\\[-1ex]
\hspace*{\fill} $\square$

We now proceed with the proof of Lemma~\ref{cornerstoneproofprop34} which requires several intermediate results. The first one is rather intuitive.

\begin{lemma}\label{cubeslice}
Let $C>0$ be a real number and $p\ge 1$ be an integer. Then every section of the cube $[0, C]^p$ with a $(p-1)$--dimensional subspace of $\R^p$ has a volume at most $\sqrt{2}C^{p-1}$.
\end{lemma}

\begin{proof}
See~\cite[Theorem 4]{cubeslicing}.
\end{proof}

\begin{lemma}\label{lemma_intermed}
Let $k\ge 1$ denote an integer and let $\vv w:= (w_0, \dots, w_k)\in\R^{k+1}$. Let $\omega, B>0$ be real numbers. Furthermore, assume that the $k+1$ real numbers $0<t_0<\dots<t_k$ satisfy the following two assumptions~:
\begin{itemize}
\item[(1)] $\underset{0\le i\neq j\le k}{\min}\, \left|t_i - t_j\right|\,\ge\, \omega$,\\[0ex]
\item[(2)] $\underset{0\le i \le k}{\max}\, |t_i|^k\,\le\, B$.\\[0ex]
\end{itemize}
Then, there exist an index $0\le j \le k$ such that $$\left|\sum_{i=0}^{k} w_i t_j^i\right|\,\geq\, \phi(\omega, B; k)\cdot\left\|\vv w\right\|_2,$$ where $\phi(\omega, B; k)$ is the quantity defined in~\eqref{def_phi_deltaBk}.
\end{lemma}

The following notation will be used in the course of the proof of Lemma~\ref{lemma_intermed}~: given a point $\vv x\in \R^n$ and a set $A \subset\R^n $, $\textrm{dist}(\vv x,A)$ will denote the quantity
\begin{equation}\label{def_dist}
\textrm{dist}(\vv x,A) := \inf \{ \|\vv x-\vv a\|_2 : \vv a \in A \}.
\end{equation}

\begin{proof}
Let $\vv X:=(\vv x_1, \cdots, \vv x_{k+1})\in M_{k+1, k+1}$ denote the matrix defined by the following $k+1$ column vectors in $\R^{k+1}$~:
\begin{align*}
\vv x_1\, &:=\, (1, t_0, \dots, t_0^k)^T, \\[1ex]
&\vdots\\[1ex]
\vv x_{k+1}\, &:=\, (1, t_k, \dots, t_k^k)^T.
\end{align*}
Together with the origin, these points form a simplex $\mathcal{S}(\vv X)$ in $\R^{k+1}$ whose volume $\left|\mathcal{S}(\vv X)\right|_{k+1}$ satisfies the well--known equation $$\left|\mathcal{S}(\vv X)\right|_{k+1}\,=\,\frac{1}{(k+1)!}\det \begin{pmatrix}\vv x_1 & \dots & \vv x_{k+1} & \vv 0 \\ 1&\dots&1&1\end{pmatrix}.$$
The formula for the determinant of a Vandermonde matrix together with hypothesis (1) then yields the inequality $$\left|\mathcal{S}(\vv X)\right|_{k+1}\, \ge \, \frac{\omega^{k(k-1)/2}}{(k+1)!}\cdotp$$
Note that hypothesis (2) implies that all the vectors $\vv x_1, \cdots, \vv x_{k+1}$ lie in the hypercube $\mathcal{B}:=[0,B]^{k+1}$. As a consequence, the volume of the section of the simplex $\mathcal{S}(\vv X)$ with any hyperplane does not exceed the volume of the section of $\mathcal{S}(\vv X)$ with $\mathcal{B}$ which, from Lemma~\ref{cubeslice}, is at most $\sqrt{2}B^k$. Also, given a hyperplane $\mathcal{P}$, it should be clear that $$\left|\mathcal{S}(\vv X)\right|_{k+1}\,\le \, 2\cdot\underset{1\le j \le k+1}{\max}\textrm{dist}\left(\vv x_j, \mathcal{P}\right)\cdot\left|\mathcal{P}\cap\mathcal{S}(\vv X)\right|_{k}.$$
The upshot of this discussion is that the following inequality holds~:
\begin{align}\label{maxdist}
\underset{1\le j \le k}{\max}\, \textrm{dist}\left(\vv x_j, \mathcal{P}\right) \, \ge \, \frac{\omega^{k(k-1)/2}}{2\sqrt{2}B^k(k+1)!}:= \phi(\omega, B, k).
\end{align}
Consider now the hyperplane $\mathcal{P}=\vv w^{\perp}$ and let $j$ be one of the indices realizing the maximum in~\eqref{maxdist}. The conclusion of the lemma is then a direct consequence of the equation $$\textrm{dist}\left(\vv x_j, \vv w^{\perp}\right)\,=\,\frac{\left|\sum_{i=0}^{k}w_i t_j^i\right|}{\left\|\vv w\right\|_2}\cdotp$$
\end{proof}

The next result contains the main substance of the proof of Lemma~\ref{cornerstoneproofprop34}.

\begin{lemma}\label{mainsubslemmacornerstone}
Let $f$ be a real valued function of class $C^k$ ($k\ge 1$) defined in a neighbourhood of $(x_0, y_0)\in\R^2$. Let $c>0$ be a real number such that $$\left|\frac{\partial^k f}{\partial x^k}(x_0, y_0)\right|\ge c.$$ Then, there exist two orthonormal vectors $\vv u, \vv v \in \Sph^{1}$ such that
\begin{align*}
\min\left\{\left|\frac{\partial^k f}{\partial \vv u^k}(x_0, y_0)\right|, \left|\frac{\partial^k f}{\partial \vv v^k}(x_0, y_0)\right| \right\}\, &\ge \, \frac{c}{2^{3k/2}}\cdot \phi\left(\left( \sqrt{2}(2k)^{2+k(k-1)/2}(k+1)!\right)^{-1}, 2, k \right)\\[1ex]
& = \, c\cdot \sigma(k,2).
\end{align*}
\end{lemma}

\begin{proof}
Set $$\vv w= (w_0, \dots, w_k):= \left(\binom{k}{j} \frac{\partial^k f}{\partial x^{k-j}\partial y^j}(x_0, y_0)\right)_{0\le j\le k}\in \R^{k+1}.$$ It readily follows from the assumptions of the lemma that $$\left\|\vv w\right\|_2\ge c.$$
Let $\lambda>0$ be a real number such that, for all indices $0\le j\le k$, $$\left|\frac{\partial^k f}{\partial x^{k-j}\partial y^j}(x_0, y_0)\right|\le \lambda.$$ We thus have the inequality
\begin{align}\label{lambdaauxiliarum}
\left\|\vv w\right\|_2 \le 2^k (k+1)\lambda.
\end{align}

Define now $k+1$ real numbers $t_0, \dots, t_k$ as follows~: $$t_i := \frac{1}{2}+\frac{i}{2k},$$ where $i=0,\dots, k$.

With the choices of the parameters $\omega:=1/(2k)$ and $B=1$, Lemma~\ref{lemma_intermed} applied to the vector $\vv w$ and to the system of points $(t_i)_{0\le i \le k}$ yields the existence of a point $t_j\in [1/2, 1]$ such that
\begin{align}\label{linterinproof}
\left|\sum_{i=0}^{k}w_i t_j^i\right|\, \ge \, \left\|\vv w\right\|_2\cdot \phi\left( \frac{1}{2k}, 1, k\right).
\end{align}

Let
\begin{align}\label{linterinproofbis}
\epsilon := \frac{1}{2k}\cdot \phi\left( \frac{1}{2k}, 1, k\right)\, \le \, \frac{1}{2k}
\end{align}
denote  a constant and $$g : t\in [0,1 ]\mapsto \sum_{i=0}^{k}w_i t^i$$ a function. Note that for all $t\in [0, 1]$, $$\left|�g'(t)\right| = \left|�\sum_{i=1}^{k}w_i i t^{i-1}\right|\, \le \, 2^k\lambda \cdot \frac{k(k+1)}{2}\, = \, 2^{k-1}\lambda k(k+1).$$

This implies that for all $t\in [t_j-\epsilon, t_j+\epsilon]$, where $t_j$ is the constant appearing in~\eqref{linterinproof}, the following inequalities hold~:
\begin{align*}
\left|g(t)\right| = \left|\sum_{i=0}^{k}w_i t^i�\right| &\ge\left|g(t_j)\right| - \left|g(t)-g(t_j)\right|\\[1ex]
&\ge\left|g(t_j)\right| -\epsilon \cdot 2^{k-1}\lambda k(k+1)\\[1ex]
&\underset{\eqref{linterinproof} \& \eqref{linterinproofbis}}{\ge} \phi\left( \frac{1}{2k}, 1, k\right)\cdot \left( \left\|\vv w\right\|_2 - 2^{k-2}\lambda (k+1)\right)\\[1ex]
&\underset{\eqref{lambdaauxiliarum}}{\ge} \frac{\left\|\vv w\right\|_2}{2}\cdot \phi\left( \frac{1}{2k}, 1, k\right).
\end{align*}

Consider now the image $[a, b]\subset [1, 2]$ of the interval $[t_j-\epsilon, t_j+\epsilon]\cap [1/2,1]$ under the map $t\mapsto 1/t$. It is then readily verified that $$|b-a|\,\ge \,\epsilon.$$ With the choices of the parameters $$\omega := \frac{\epsilon}{k}\, =\, \frac{1}{2k^2}\cdot\phi\left(\frac{1}{2k}, 1, k \right)$$ and $B=2$, apply once more Lemma~\ref{lemma_intermed}, this time to the vector $\left( (-1)^i w_i \right)_{0\le i \le k}$ and to the set of points $$\tilde{t}_i = a+\frac{b-a}{k}\cdot i,  \qquad 0\le i \le k.$$ This yields the existence of $\tilde{t}_j\in [a, b]$ such that $$\left|\sum_{i=0}^{k} (-1)^i w_i \tilde{t}_j^i�\right|\, \ge \, \left\|\vv w\right\|_2\cdot \phi\left(\frac{\epsilon}{k}, 2, k\right).$$

The upshot of this is that, when considering the point $s:= 1/\tilde{t}_j$, the following two inequa\-li\-ties hold simultaneously~:
\begin{align*}
 \left|\sum_{i=0}^{k}w_i s^i�\right|\, &\ge \,  \frac{\left\|\vv w\right\|_2}{2}\cdot \phi\left( \frac{1}{2k}, 1, k\right)
\end{align*}
and
\begin{align*}
\left|\sum_{i=0}^{k} (-1)^i w_i s^{-i}�\right|\, &\ge \, \left\|\vv w\right\|_2\cdot \phi\left(\frac{1}{2k^2}\cdot\phi\left(\frac{1}{2k}, 1, k\right), 2, k\right).
\end{align*}
Since $s\in [1/2,1]$, it is easily seen that one can find a unit vector $(u_1, u_2)\in\Sph^1$ such that $s=u_2/u_1$ and $u_1, u_2\in [1/(2\sqrt{2}), 1]$. Let $\vv u\in\Sph^1$ and $\vv v\in\Sph^1$ denote the two orthonormal vectors defined as $\vv u:=(u_1, u_2)$ and $\vv v:=(u_2, -u_1)$.

Note then that
\begin{align*}
\left|\sum_{i=0}^{k}w_i u_1^{k-i}u_2^i�\right|\, = \, u_1^k\left|\sum_{i=0}^{k}w_is^i\right|\, &\ge \, \frac{1}{2^{1+3k/2}}\cdot\left\|\vv w\right\|_2\cdot \phi\left(\frac{1}{2k}, 1, k\right)\\[1ex]
&\ge \, \frac{c}{2^{1+3k/2}}\cdot \phi\left(\frac{1}{2k}, 1, k\right)
\end{align*}
and, similarly,
\begin{align*}
\left|\sum_{i=0}^{k}(-1)^iw_i u_2^{k-i}u_1^i�\right|\, = \, u_2^k\left|\sum_{i=0}^{k}(-1)^i w_i s^{-i}\right|\, &\ge \, \frac{c}{2^{3k/2}}\cdot\phi\left(\frac{1}{2k^2}\cdot\phi\left(\frac{1}{2k}, 1, k\right), 2, k\right).
\end{align*}
Since, from the definition of the vector $\vv w$, $$\frac{\partial^k f}{\partial \vv u^k}(x_0,y_0)\,=\,\sum_{i=0}^{k}w_i u_1^{k-i}u_2^i�$$ and $$\frac{\partial^k f}{\partial \vv v^k}(x_0,y_0)\,=\, \sum_{i=0}^{k}(-1)^iw_i u_2^{k-i}u_1^i,$$ this completes the proof of the lemma from the definition of  $\phi$ in~\eqref{def_phi_deltaBk}.

\end{proof}

We now have all the ingredients at our disposal to prove Lemma~\ref{cornerstoneproofprop34}.

\noindent\textit{Proof of Lemma~\ref{cornerstoneproofprop34}.} Denote the coordinates of the vector $\vv x\in\R^d$ as $\vv x = (x_1, \dots, x_d)$. Even if it means relabeling the axes, assume furthermore without loss of generality that $i_0=1$ in the statement of the lemma. The proof then goes by induction on $d\ge 1$, the conclusion being trivial when $d=1$. When $d=2$, Lemma~\ref{cornerstoneproofprop34} reduces to Lemma~\ref{mainsubslemmacornerstone}. Assume therefore that $d\ge 3$. It then readily follows from the induction hypothesis applied to the function $(x_1, \dots, x_{d-1})\in\R^{d-1}\mapsto f(x_1, \dots, x_{d-1}, x_d)$ that there exists a rotation $S_1 : \R^d \rightarrow \R^d$ such that $$\left|\frac{\partial^k (f\circ S_1)}{\partial x_i^k}(\vv x)\right|\,\ge \, \mu\cdot \sigma(k, d-1).$$
Consider now the function $(x_1, x_d)\in\R^2\mapsto f(x_1, \dots, x_{d-1}, x_d)$. Applying Lemma~\ref{mainsubslemmacornerstone} to this function with $c=\mu\cdot\sigma(k,d-1)$ therein provides the existence of a rotation $S_2 : \R^d\mapsto \R^d$ acting on the plane $(x_1, x_d)$ and leaving its orthogonal unchanged such that $$\min \left\{\frac{\partial (f\circ S_1 \circ S_2)}{\partial x_1^k}(\vv x), \, \frac{\partial (f\circ S_1 \circ S_2)}{\partial x_d^k}(\vv x) \right\}\,\ge \, \mu\cdot \sigma(k,d-1)\cdot\sigma(k,2)\,=\,\mu\cdot\sigma(k,d).$$
The lemma follows upon setting $S=S_1\circ S_2$.\hspace*{\fill} $\square$

\subsection{Higher rank case of condition (i)}  The key to verifying condition (i) of Theorem~\ref{theo_BKM} in the  case when $\Gamma$ is of rank greater than one is  Proposition \ref{proposition_explicit_4_1} below. In short, it is  an explicit version of~\cite[Proposition 4.1]{Bernik-Kleinbock-Margulis-01:MR1829381} in the particular case when the set $\mathcal{G}$ appearing therein is given by
\begin{equation}\label{defg}
\mathcal{G}:=\left\{\left(\vv u_1 \cdot \vv f , \vv u_2 \cdot \vv f+u_0\right) \; : \; u_0\in\R,\, \bm{u}_1, \bm{u}_2\in\R^n, \, \bm{u}_1\perp\bm{u}_2 \right\}.
\end{equation}
The statement is concerned with the \emph{skew gradient} of a map as defined  in~\cite[\S4]{Bernik-Kleinbock-Margulis-01:MR1829381}. We recall the definition.
Let $\vv g=(g_1,g_2):\tcUp\rightarrow\R^2$ be a differentiable function. The skew gradient $\widetilde{\nabla}\vv g:\tcUp\rightarrow\R^2 $ is defined by
$$
\widetilde{\nabla}\vv g(\vv x):=g_1(\vv x)\nabla g_2( \vv x)-g_2(\vv x)\nabla g_1(\vv x).
$$
If we write   $\vv g(\vv x)$ in terms of polar coordinates; i.e.~via the usual functions $\rho(\vv x)$ and $\theta(\vv x)$, it is then readily verified that
\begin{align}\label{latest}
\widetilde{\nabla}\vv g(\vv x) = \rho^2(\vv x) \nabla \theta(\vv x).
\end{align}
Essentially, the skew gradient measures how different  the pair of  functions $g_1 $ and $g_2$ are from being proportional to each other.

\begin{proposition}\label{proposition_explicit_4_1}
Let $\cU\subset\R^d$ be a ball and  $\vv f=(f_1,\dots,f_n)$ be an $n$--tuple of $C^{l+1}$ functions satisfying  Assumption 1. Let $\rho_2$, $C_{d,l}$ and $\cG$  be  given by $\eqref{def_rho_two_first} $, $ \eqref{def_Cdl}$ and $ \eqref{defg} $ respectively.
Then,
\begin{itemize}

\item[(a)~] for all $\bm{g}\in\mathcal{G}$, $$\|\tilde{\nabla}\bm{g}\|_2  \quad  is \quad  \left(2C_{d,l}, \frac{1}{d(2l-1)} \right)-good \ on \ \tcU$$  \\[2ex]
\item[(b)~] for all $\bm{g}\in\mathcal{G}$,
\begin{equation}\label{inegrhogb}
\underset{\bm{x}\in \cU}{\sup} \; \| \tilde{\nabla}\bm{g}(\bm{x})\|_2 \, \ge \, \rho_2 \, .
\end{equation}

\end{itemize}
\end{proposition}

\bigskip

This proposition together with  Corollary~\ref{explicoro3.5} and the basic properties of $(C, \alpha)$--good functions given in \cite[Lemma~3.1]{Bernik-Kleinbock-Margulis-01:MR1829381}  enables us to deduce the following statement, which establishes condition (i)  in the higher rank case.

\begin{corollary}
Let $\cU\subset\R^d$ be a ball and  $\vv f=(f_1,\dots,f_n)$ be an $n$--tuple of $C^{l+1}$ functions satisfying  Assumption 1.
Let $\Lambda$ be the discrete subgroup given  by~\eqref{def_Lambda} and  $\Gamma\in\cL(\Lambda)$ be a primitive subgroup of $\Lambda$.    Furthermore, let  $H$ be the map given by~\eqref{def_H}. Then, the function
\begin{equation} \label{ez1}
\vv x \mapsto \nu_{*} \left(H(\vv x)\Gamma\right)
\end{equation}
is $\left(C, \alpha\right)$--good on the ball $\tcU$ with constants $C$ and $\alpha$ given by~\eqref{def_C} and~\eqref{def_alpha} respectively.
\end{corollary}
\begin{proof}
Let $k$ denote the rank of $\Gamma $.   The case $k=1$ has already been established as a consequence of Corollary~\ref{explicoro3.5} in \S\ref{900}.  Assume therefore that $ k \geq 2$. It is shown  in~\cite[\S7 Eq(7.3)]{Bernik-Kleinbock-Margulis-01:MR1829381}  that there exist real numbers $a,b,\mu\in\R$ such that, for all $\vv x\in \tcU$,  $\nu_{*} \left(H(\vv x  )\Gamma\right)  $ given by  \eqref{ez1} can be expressed as the Euclidean norm of a vector $\vv w (\vv x)$. Furthermore, there exists an orthonormal system of vectors of the form $\mathcal{S}=\left\{\vv e_0, \vv e_1^*, \dots, \vv e_d^*, \vv v_1, \dots, \vv v_{k-1} \right\}$ when $k\le n$ or of the form $\mathcal{S}=\left\{\vv e_0, \vv e_1^*, \dots, \vv e_d^*, \vv v_0, \dots, \vv v_{k-1} \right\}$ when $k=n+1$ such that $\vv w (\vv x)$ is a linear combination of $$L_d(k):= \frac{(k+1)(dk+2)}{2}$$ skew products of elements of $\mathcal{S}$ whose coefficients are of any of the following form~:
\begin{align}
\label{ez_class_1} a+b\vv f\cdot \vv v_0 \qquad \\[1ex]
\label{ez_class_2} b \qquad \\[1ex]
\label{ez_class_3} b\,\vv f\cdot \vv v_i\qquad  & (1\le i\le k-1)\\[1ex]
\label{ez_class_4} b\,\mu\,\partial_s(\vv f\cdot\vv v_i)\qquad  & (1\le i \le k, \; 1\le s \le d)\\[1ex]
\label{ez_class_5} \mu \, X(i,s)  \qquad &(1\le i\le k-1,\; 1\le s \le d)\\[1ex]
\label{ez_class_6} b\,\mu \,Y(i,j,s)  \qquad &(1\leq i < j\leq k-1,\; 1\le s \le d),
\end{align}
where $$X(i,s):= \left(\vv f\cdot\vv v_i\right) \partial_s\left(a+b\vv f\cdot\vv v_0\right) - \left(a+b\vv f\cdot\vv v_0\right) \partial_s \left(\vv f\cdot\vv v_i\right)$$ and $$Y(i,j,s):= \left(\vv f\cdot\vv v_i\right) \partial_s\left(\vv f\cdot\vv v_j\right) -  \left(\vv f\cdot\vv v_j\right) \partial_s\left(\vv f\cdot\vv v_i\right).$$

\medskip

\begin{enumerate}
\item[$\bullet$]
It follows from part (a) of Corollary~\ref{explicoro3.5} and \cite[Lemma~3.1(a,d)]{Bernik-Kleinbock-Margulis-01:MR1829381} that the  coordinate functions given by ~\eqref{ez_class_1},~\eqref{ez_class_2} and~\eqref{ez_class_3}   are $(C',\alpha)$--good, where
$$
C':= \frac{C}{\left(L_d(n+1)\cdot d\right)^{\alpha/2}}\cdotp
$$
\item[]
\item[$\bullet$]It follows from part (b) of  Corollary~\ref{explicoro3.5} and  \cite[Lemma~3.1(a,d)]{Bernik-Kleinbock-Margulis-01:MR1829381} that, when the index $i$ is fixed, the maximum over $s$ of the coordinate functions  given by~\eqref{ez_class_4}, that is, the quantity $\|b\, \mu\, \nabla(\vv f\cdot\vv v_i)\|_{\infty}$,  is  $(C',\alpha)$--good.
\item[]
\item[$\bullet$] It  follows from Proposition~\ref{proposition_explicit_4_1} and \cite[Lemma~3.1(a,d)]{Bernik-Kleinbock-Margulis-01:MR1829381} that, for fixed indices $i$ and $j$, the Euclidean norm over $s$ of the coordinate functions  given by~\eqref{ez_class_5} and~\eqref{ez_class_6}, that is, the quantities $\| \mu\, \tilde{\nabla}(\vv f\cdot\vv v_i,  a+b\vv f\cdot\vv v_0)\|_{2}$ and $\|b\, \mu\, \tilde{\nabla}(\vv f\cdot\vv v_i, \vv f\cdot\vv v_j)\|_{2}$ respectively, are $(C',\alpha)$--good. On using the relation $$\frac{1}{\sqrt{d}}\leq\frac{\|\cdot\|_{\infty}}{\|\cdot\|_2}\leq 1$$ valid in $\R^d$ and~\cite[Lemma~3.1(c)]{Bernik-Kleinbock-Margulis-01:MR1829381}, it follows that $\| \mu\, \tilde{\nabla}(\vv f\cdot\vv v_i,  a+b\vv f\cdot\vv v_0)\|_{\infty}$ and $\|b\, \mu\, \tilde{\nabla}(\vv f\cdot\vv v_i, \vv f\cdot\vv v_j)\|_{\infty}$ are $(d^{\alpha/2}C',\alpha)$--good.
\end{enumerate}

The upshot of the above together with \cite[Lemma~3.1(b)]{Bernik-Kleinbock-Margulis-01:MR1829381} is that the maximum of the coordinate functions \eqref{ez_class_1}--\eqref{ez_class_6} is $(d^{\alpha/2}C',\alpha)$--good. In turn,  on using the relation
$$
\frac{1}{\sqrt{L_d(k)}}\leq\frac{\|\cdot\|_{\infty}}{\|\cdot\|_2}\leq 1
$$
valid in $\R^{L_d(k)}$ and~\cite[Lemma~3.1(c)]{Bernik-Kleinbock-Margulis-01:MR1829381}, we have that
$$ \nu_{*} \left(H(\, . \,  )\Gamma\right)  \  {\rm  \ is \ }  \left(C'(d\cdot L_d(k))^{\alpha/2},\alpha\right)-{\rm good}.
$$
As $k\leq n+1$, the desired statement follows.

\end{proof}

Modulo the proof of Proposition~\ref{proposition_explicit_4_1}, we have completed the task of verifying
condition (i) of Theorem~\ref{theo_BKM}.  The proof of the proposition  is rather lengthy and therefore is  postponed till after we have verified condition (ii) of Theorem~\ref{theo_BKM}.

\subsection{Verifying  condition (ii) of Theorem~\ref{theo_BKM}  modulo Proposition~\ref{proposition_explicit_4_1}}

The following lemma, which although not explicitly stated, is essentially proved in~\cite[\S7]{Bernik-Kleinbock-Margulis-01:MR1829381}, see \cite[Eq(7.5)]{Bernik-Kleinbock-Margulis-01:MR1829381} and onwards.
The key difference is that we make use of Proposition~\ref{proposition_explicit_4_1} in place of \cite[Proposition~4.1]{Bernik-Kleinbock-Margulis-01:MR1829381} and so  are able to give explicit values of $\rho_1$ and $\rho$.

\begin{lemma} \label{tech_lemma_one}
Let $\cU\subset\R^d$ be a ball and  $\vv f=(f_1,\dots,f_n)$ be an $n$--tuple of $C^{l+1}$ functions satisfying  Assumption 1. 
 Let $\rho_1, \rho >0$ be given  by~\eqref{def_rho_one_first} and ~\eqref{deffinalerho_first} respectively
and assume that for any $\vv v \in \Sph^{n-1}$ and $p\in\R$ we have that
\begin{equation} \label{tech_lemma_one_def_rho_one}
\sup_{\vv x\in\cU}|\vv f(\vv{x})\vv\cdot\vv v+p|\ge\rho_1\; \text{ and } \;\sup_{\vv x\in\cU}\left\|\nabla\left(\vv f(\vv{x})\vv\cdot\vv v\right)\right\|_{\infty}\ge\rho_1 .
\end{equation}
Furthermore, let $\Lambda$ be the discrete subgroup given  by~\eqref{def_Lambda},  $\Gamma\in\cL(\Lambda)$ be a primitive subgroup of $\Lambda$ and   $H$ be the map given by~\eqref{def_H}.
Then
\[
\sup_{\vv x\in\cU}\nu_{*}(H(\vv x)\Gamma)\ge\rho .
\]

\end{lemma}

The following statement immediately verifies condition (ii) of Theorem~\ref{theo_BKM}. It is the above lemma without the  assumptions made in (\ref{tech_lemma_one_def_rho_one}).

\medskip

\begin{corollary} \label{tech_corollary_one}
Let $\cU\subset\R^d$ be a ball and  $\vv f=(f_1,\dots,f_n)$ be an $n$--tuple of $C^{l+1}$ functions satisfying  Assumption 1.
Let $\Lambda$ be the discrete subgroup given  by~\eqref{def_Lambda} and  $\Gamma\in\cL(\Lambda)$ be a primitive subgroup of $\Lambda$.    Furthermore, let  $H$ be the map given by~\eqref{def_H}.   Then
\begin{equation} \label{tech_corollary_one_result}
\sup_{\vv x\in \cU}\nu_{*}(H(\vv x)\Gamma)\ge\rho,
\end{equation}
where $\rho$ is given by~\eqref{deffinalerho_first}. 
\end{corollary}

\medskip

\begin{proof}[Proof of Corollary~\ref{tech_corollary_one}]
The desired statement  follows directly  from Lemma  \ref{tech_lemma_one} on ve\-ri\-fy\-ing  the inequalities associated with  (\ref{tech_lemma_one_def_rho_one}).
Let $\vv v \in \Sph^{n-1}$. By the definition of $s_0:=s(l;\vv f,\tcUp)$, there exists  a $\vv u \in \Sph^{d-1}$ and $1\le k\le l$ such that
\begin{equation} \label{tech_corollary_one_lb_def_s}
\left|\frac{\partial^k(\vv f\cdot\vv v)}{\partial\vv u^k}(\vv x_0)\right|\ge s_0.
\end{equation}
Recall, that $\vv x_0$  is the centre of $\cU$. It follows that for any  $\vv x\in\cU$, we have that~:
\begin{align}
\label{minvgksx}
\left|\bm{v}\cdotp\frac{\partial^k \bm{f}}{\partial\bm{u}^k}(\bm{x}) \right|
& = \, \left|\bm{v}\cdotp\frac{\partial^k \bm{f}}{\partial\bm{u}^k}(\vv x_0) \right| - \left|\bm{v}\cdotp\left(\frac{\partial^k \bm{f}}{\partial\bm{u}^k}(\vv x_0) - \frac{\partial^k \bm{f}}{\partial\bm{u}^k}(\bm{x})\right)\right| \nonumber\\[2ex]
&\ge \, s_0 - \left\| \frac{\partial^k \bm{f}}{\partial\bm{u}^k}(\vv x_0) - \frac{\partial^k \bm{f}}{\partial\bm{u}^k}(\bm{x})\right\|_{2}.
\end{align}
Let $\bm{s}'$ denote  the unit vector $$\bm{s}':= \frac{\bm{x} - \vv x_0}{\left\|\bm{x} - \vv x_0\right\|_2}\cdotp $$
By Lagrange's Theorem, there exists $\bm{x}'$ between $\vv x_0$ and $\bm{x}$ such that
\[
\frac{\partial^k \bm{f}}{\partial\bm{u}^k}(\vv x_0) \, = \, \frac{\partial^k \bm{f}}{\partial\bm{u}^k}(\bm{x}) + \left\|\bm{x}- \vv x_0\right\|_2\frac{\partial}{\partial \bm{s}'}\left(\frac{\partial^k \bm{f}}{\partial\bm{u}^k}\right)(\bm{x}').
\]
It then follows from~(\ref{minvgksx}) and the definition of $M$ in~\eqref{eqn:def_N} that
\[
\left|\bm{v}\cdotp\frac{\partial^k \bm{f}}{\partial\bm{u}^k}(\bm{x}) \right|\, \ge \, s_0 - M\left\|\bm{x}- \vv x_0\right\|_{2} \, .
\]
This together with the fact that $r<s_0/2M$  ---  a direct consequence of (\ref{slv2}) ---, implies that
\begin{equation} \label{tech_corollary_one_lb_one}
\left|\frac{\partial^k(\vv f\cdot\vv v)}{\partial\vv u^k}(\vv x)\right|\ge s_0/2   \qquad \forall \  \vv x\in\cU  \, .
\end{equation}

\noindent  The upshot is that the hypotheses of~\cite[Lemma~3.6]{Bernik-Kleinbock-Margulis-01:MR1829381} are satisfied.  A straightforward  application  of~\cite[Lemma~3.6]{Bernik-Kleinbock-Margulis-01:MR1829381} together with~\eqref{def_rho_one_first} implies that
\[
\sup_{x\in\cU}\left|\vv f(\vv x)\cdot\vv v+p\right|\ge\frac{s_0/2}{2k^k(k+1)!}(2r)^k\ge\frac{s_0}{4l^l(l+1)!}(2r)^l=\rho_1\sqrt{d}\geq\rho_1
\]
for any $p\in\R$.
Thus the first inequality appearing in  ~\eqref{tech_lemma_one_def_rho_one} is established.

It remains to prove the second inequality in~\eqref{tech_lemma_one_def_rho_one}; that is, that for any $\vv v \in \Sph^{n-1}$,
\begin{equation} \label{tech_corollary_one_aim_two}
\sup_{\vv x\in\cU}\left\|\nabla\left(\vv f(\vv x)\vv \cdot\vv v\right)\right\|_{\infty}\ge\rho_1=\frac{s_0}{4l^l(l+1)!\sqrt{d}}(2r)^l.
\end{equation}
Recall from above that  for any  $\vv v \in \Sph^{n-1}$ we can find a vector  $\vv u \in \Sph^{d-1}$ such that~\eqref{tech_corollary_one_lb_def_s} and~\eqref{tech_corollary_one_lb_one} hold. Furthermore, observe that
\begin{equation} \label{tech_corollary_one_equality_one}
\frac{\partial(\vv f\cdot\vv v)}{\partial\vv u}(\vv x)=\vv u^t\cdot\nabla\left(\vv f(\vv x)\cdot\vv v\right).
\end{equation}

\noindent We proceed by considering two cases, depending on whether or not $k=1$ in ~\eqref{tech_corollary_one_lb_def_s}.

\begin{enumerate}
\item[$\bullet$] Suppose $k=1$ in~\eqref{tech_corollary_one_lb_def_s}. Then it follows from  ~\eqref{tech_corollary_one_lb_one} and~\eqref{tech_corollary_one_equality_one}  that
\[
\left|\vv u^t\cdot\nabla\left(\vv f(\vv x)\cdot\vv v\right)\right|\ge \frac{s_0}{2}  \qquad  \forall  \ \vv x\in\cU  \, .
\]
On applying the  Cauchy--Schwartz inequality,  we obtain that
\[
\left\|\nabla\left(\vv f(\vv x)\cdot\vv v\right)\right\|_{2}\ge \frac{s_0}{2}  \qquad  \forall  \ \vv x\in\cU  \, .
\]
This together with the fact that $\left\|\, . \,\right\|_{2}\le \sqrt{d}\left\|\, . \,\right\|_{\infty}$  implies the  second inequality appearing in  ~\eqref{tech_lemma_one_def_rho_one}.

\item[]

\item[$\bullet$] 
Suppose $k\ge 2$ in~\eqref{tech_corollary_one_lb_def_s}. Consider the function
$g(\vv x):=\frac{\partial(\vv f\cdot\vv v)}{\partial\vv u}(\vv x)$ defined on $\cU$. Then by~\eqref{tech_corollary_one_lb_one}, we have that
\[
\left|\frac{\partial^{k-1}g}{\partial\vv u^{k-1}}(\vv x)\right|\ge \frac{s_0}{2} \qquad  \forall  \ \vv x\in\cU  \, .
\]
 Thus,  the hypotheses of~\cite[Lemma~3.6]{Bernik-Kleinbock-Margulis-01:MR1829381} are satisfied for the function $g(\vv x )$ and  a straightforward  application  of that lemma together with~\eqref{def_rho_one_first} implies that
\begin{equation} \label{tech_corollary_one_ie_two}
\sup_{\vv x\in\cU}|g(\vv x)|\ge\frac{s_0}{4(l-1)^{l-1}l!}(2r)^{l-1}>\rho_1\sqrt{d}.
\end{equation}
Now  the  Cauchy--Schwartz inequality and~\eqref{tech_corollary_one_equality_one} imply that
$$
\left\|\nabla\left(\vv f(\vv x)\cdot\vv v\right)\right\|_{2}\ge\left|\vv u^t\cdot\nabla\left(\vv f(\vv x)\cdot\vv v\right)\right| =  |g(\vv x )|   \qquad  \forall  \ \vv x\in\cU  \,    .
$$
This together with ~\eqref{tech_corollary_one_ie_two} and the fact that $\left\|\, . \,\right\|_{2}\le \sqrt{d}\left\|\, . \,\right\|_{\infty}$ imply the desired statement; namely that
$$
\sup_{\vv x\in\cU}\left\|\nabla\left(\vv f(\vv x)\cdot\vv v\right)\right\|_{\infty}\ge \rho_1 \ .
$$
\end{enumerate}
\end{proof}

The upshot of \S\ref{cond1-3} is that we have verified conditions (i) \& (ii) of Theorem~\ref{theo_BKM}  as desired  modulo  Proposition~\ref{proposition_explicit_4_1}.

\section{Proof of Proposition~\ref{proposition_explicit_4_1} \label{xmas2}}

In order to prove Proposition~\ref{proposition_explicit_4_1}, we first establish  an explicit version of~\cite[Lemma 4.3]{Bernik-Kleinbock-Margulis-01:MR1829381}. Throughout this section, the notation introduced in~\eqref{def_dist} will be used.


\medskip

\begin{lemma}
\label{explilem4.3}
Let $B\subset \R^d$ be a ball of radius 1 and let $B_{\infty}$ denote the hypercube circumscribed around $B$ with edges parallel to the coordinate axes. Assume further that $\bm{p}=(p_1, p_2)~: B \mapsto \R^2$ is a polynomial map of degree at most $l\ge 1$ such that
\begin{equation}\label{diamimage}
\underset{\bm{x}, \bm{y} \in B_{\infty}}{\sup} \left\|\bm{p}(\bm{x}) - \bm{p}(\bm{y}) \right\|_2 \, \le \, 2
\end{equation}
and
\begin{equation}\label{distdroite}
\underset{\bm{x}\in B}{\sup}\; \textrm{\emph{dist}} \left(\mathcal{L}, \bm{p}(\bm{x}) \right) \, \ge \, \frac{1}{8}
\end{equation}
for any straight line $\mathcal{L}\subset \R^2$.
Then,
\begin{equation} \label{a}
\underset{\bm{x}\in B}{\sup} \|\widetilde{\nabla}\bm{p}(\bm{x}) \|_2 \, \ge \, \frac{1}{86\,016\,\sqrt{10}}\left(1+\underset{\bm{x}\in B}{\sup} \left\|\bm{p}(\bm{x})\right\|_2 \right)
\end{equation}
and
\begin{equation}
	\label{b}
\underset{\bm{x}\in B, \, i=1,2}{\sup} \left\|\nabla p_i(\bm{x}) \right\|_2 \, \le \, 2l^2\sqrt{d}.
\end{equation}
\end{lemma}

\begin{proof}
Regarding~\eqref{a}, if we assume that ${\sup}_{\bm{x} \in B} \left\|\bm{p}(\bm{x})\right\|_2 > 6$,   the argument used to prove  ~\cite[Lemma 4.3]{Bernik-Kleinbock-Margulis-01:MR1829381}  gives the stronger inequality in which  the constant factor   $1/(86\,016\sqrt{10})$ is replaced by $1/64$. Thus, without loss of generality, assume that
\begin{equation} \label{dec1}
{\sup}_{\bm{x} \in B} \left\|\bm{p}(\bm{x})\right\|_2 \le 6 \, . \end{equation}   It is easily inferred from~(\ref{distdroite}) that there exists  $\bm{x}_1\in \overline{B}$, the closure of $B$, such that $\left\|\bm{p}\left(\bm{x}_1\right)\right\|_2 \ge 1/8$.
Working in polar coordinates and choosing the straight line $\mathcal{L}_1$ joining the origin to $\bm{p}\left(\bm{x}_1\right)$ to be  the polar axis, let $\left(\rho\!\left(\bm{x}\right), \theta\!\left(\bm{x}\right)\right)$ denote the polar coordinates of a vector $\bm{x}\in\R^2$. Thus, $\rho(\bm{p}(\bm{x}_1))\ge 1/8$.
Furthermore, from~(\ref{distdroite}), there exists $\bm{x}_2\in\overline{B}$ such that $\mbox{dist}\left(\mathcal{L}_1, \bm{p}(\bm{x}_2)\right)\ge 1/8$ 
 and therefore, together with~(\ref{dec1}), we have that
\begin{align}
\label{dec2}
\left|\theta\left(\bm{p}\left(\bm{x}_2\right)\right) \right| \, & \ge \, \left|\sin \theta\left( \bm{p}\left(\bm{x}_2\right)\right) \right| \, = \, \frac{\mbox{dist}\left(\mathcal{L}_1, \bm{p}(\bm{x}_2)\right)}{\rho\left(\bm{p}(\bm{x}_2)\right)}  \nonumber \\
& \ge \, \frac{1/8}{6} \, =\, \frac{1}{48}\cdotp
\end{align}

\noindent Now let $\Delta$ be the straight line joining $\bm{p}(\bm{x}_1)$ and $\bm{p}(\bm{x}_2)$. Furthermore, let  $\mathcal{L}_2$ denote the $x$-coordinate axis, $(x_1, y_1)$ the Cartesian coordinates of $\bm{p}(\bm{x}_2)$ and $(\rho(\bm{p}(\bm{x}_1)),0)$ the Cartesian coordinates of $\bm{p}(\bm{x}_1)$. Then the Cartesian equation of $\Delta$ is
\[
\Delta~: \, y_1 x -(x_1-\rho(\bm{p}(\bm{x}_1)))y - \rho(\bm{p}(\bm{x}_1))y_1 =0.
\]
It follows from the choice of the points $\bm{x}_1$ and $\bm{x}_2$ together with~(\ref{diamimage}), (\ref{distdroite}) and (\ref{dec1}) that  $$\frac{1}{8}\le \left|y_1\right| \le 6,  \   \; \rho(\bm{p}(\bm{x}_1))\ge \frac{1}{8}    \ \; \mbox{ and }   \  \; \left|x_1 - \rho(\bm{p}(\bm{x}_1))\right|\le 2.$$ Therefore, the distance from the origin $O$ to $\Delta$ satisfies the inequality
\begin{equation} \label{dec3}
\mbox{dist}\left(\Delta, O \right) \, = \, \frac{\left|(\rho(\bm{p}(\bm{x}_1)) y_1\right|}{\sqrt{y_1^2+\left(x_1 - \rho(\bm{p}(\bm{x}_1) \right)^2}}\, \ge \, \frac{(1/8)^2}{\sqrt{6^2+2^2}}\, = \, \frac{1}{128\sqrt{10}}\cdotp \end{equation}

\medskip

Let $J$ denote the straight line segment $[\bm{x}_1, \bm{x}_2]$ and let $\bm{u} $ be the unit vector $$\bm{u}:= \frac{\bm{x}_2 - \bm{x}_1}{\left\|\bm{x}_2 - \bm{x}_1\right\|_2}  \cdotp $$  Restricting $\bm{p}$ to $J$, Lagrange's Theorem guarantees the existence of $\bm{y}\in \left(\bm{x}_1, \bm{x}_2\right)$ such that $$\theta\left(\bm{p}(\bm{x}_2)\right) = \frac{\partial \theta}{\partial \bm{u}}\left(\bm{y}\right)\left| J \right|.$$
It then follows via~(\ref{latest}), (\ref{dec2}) and~(\ref{dec3})  that

\begin{align*}
\|\tilde{\nabla}\bm{p}(\bm{y})\|_2 \, &\ge \, |\bm{u}\cdot\tilde{\nabla}\bm{p}(\bm{y}) |\, = \, \rho^2\left(\bm{y}\right)\left|\frac{\partial \theta}{\partial \bm{u}}\left(\bm{y}\right)\right|\\[3ex]
&\ge \, \mbox{dist}\left(\Delta, O\right) \frac{\left| \theta(\bm{p}(\bm{x}_2))\right|}{\left|J\right|}\, =\, \mbox{dist}\left(\Delta, O\right) \frac{\left| \theta(\bm{p}(\bm{x}_2))\right|}{\left\|\bm{x}_2 -\bm{x}_1\right\|_2}\\[3ex]
&\ge \, \frac{1}{128\sqrt{10}\times 48 \times 2}\, =\, \frac{1}{12\,288\,\sqrt{10}}\cdotp
\end{align*}
Thus,
$$\underset{\bm{x}\in B}{\sup} \|\tilde{\nabla}\bm{p}(\bm{x}) \|_2 \, \ge \, \frac{1}{12\,288\,\sqrt{10}} = \frac{7}{86\,016\, \sqrt{10}} \, \ge \, \frac{1}{86\,016\, \sqrt{10}}\left(1+ \underset{\bm{x}\in B}{\sup}\left\|\bm{p}(\bm{x})\right\|_2\right).$$

\noindent This completes the proof of  (\ref{a}).  We now turn out attention to ~\eqref{b}.

Let $i\in\left\{1,2\right\}$. It may be assumed without loss of generality that $p_i\left(0, \dots, 0\right)=0$ and that the ball $B$ is centered at the origin.  Then, for given  $x_2, \dots, x_d$ in $\R$, consider the polynomial in one variable $p\left(x\right):= p_i\left(x, x_2, \dots, x_d \right)$, which is of degree at most $l$.  It follows from~(\ref{diamimage}) that
\[
\underset{\left|x\right|\le 1}{\sup} \left|p\left(x\right) \right|\,\le\, 2.
\]
Hence by  Markov's inequality for polynomials, we have that
\[
\underset{\left|x\right|\le 1}{\sup} \left|\frac{\textrm{d} p}{\textrm{d} x}\left(x\right) \right| \, = \, \underset{\left|x_1\right|\le 1}{\sup} \left|\frac{\partial p_i}{\partial x_1}\left(x_1, x_2, \dots, x_d\right)\right|\,\le\, 2l^2  \, .
\]
 This together with the fact that $\left\|\, . \,\right\|_{2}\le \sqrt{d}\left\|\, . \,\right\|_{\infty}$  implies that
\[
\max \left\{ \underset{\bm{x}\in B}{\sup} \left\|\nabla p_1(\bm{x}) \right\|_2 , \underset{\bm{x}\in B}{\sup} \left\|\nabla p_2(\bm{x}) \right\|_2 \right\} \, \le \, 2l^2\sqrt{d}
\]
and therefore completes the proof of the lemma.
\end{proof}

\bigskip

We now have all the ingredients in place  to prove Proposition~\ref{proposition_explicit_4_1}.

\medskip

\begin{proof}[Proof of Proposition~\ref{proposition_explicit_4_1}.]  The proposition is an explicit version of  \cite[Proposition 4.1]{Bernik-Kleinbock-Margulis-01:MR1829381}.  Within our setup in which $\mathcal{G}$ is given by \eqref{defg}, the starting point for the proof of part~(a) of \cite[Proposition 4.1]{Bernik-Kleinbock-Margulis-01:MR1829381} corresponds to the existence of positive constants $\delta$, $c$, and $\alpha$ with

\begin{equation} \label{tree}
0<\delta<1/8  \qquad {\rm and }  \qquad 2C_{d,l}N_d\delta^{1/(d(2l-1)(2l-2))}\le 1
\end{equation}

\noindent such that for every  $ \vv g \in\mathcal{G} $  one has

\begin{equation} \label{eq_4_5a}
\forall \, \vv v\in\Sph^{1} \quad \exists \, \vv u\in\Sph^{d-1}\quad \exists \, k\leq l  \; \; :\;\; \inf_{\vv x\in\tcU}\left|\vv v\cdot\frac{\partial^k \vv g}{\partial \vv u^k}(\vv x)\right|\geq c
\end{equation}
and
\begin{equation} \label{eq_4_5b}
\sup_{\vv x, \vv y\in\tcU}\|\partial_{\beta}\vv g(\vv x)-\partial_{\beta}\vv g(\vv y)\|\leq\frac{\delta c \alpha}{8 \xi l^l (l+1)!}\,=\,  \frac{\delta c \alpha}{16 l^{l+2} (l+1)!\sqrt{d}}
\end{equation}
for all multi--indices $\beta$  with  $|\beta|=l$. Here, $\xi = 2l^2\sqrt{d}$ is the quantity in right--hand side of~\eqref{b} and the real number $\alpha$ is required to be less than the constant appearing in the right--hand side of~\eqref{a}, that is,
\begin{equation} \label{eq_4_5c}
\alpha\le\frac{1}{86\,016\sqrt{10}}\cdotp
\end{equation}

The statements~\eqref{eq_4_5a} and~\eqref{eq_4_5b} correspond exactly to~\cite[Eq(4.5a) \& Eq(4.5b)]{Bernik-Kleinbock-Margulis-01:MR1829381}) with $V$ replaced by $\tcU$.

The proof of part~(a) of Proposition~\ref{proposition_explicit_4_1} follows from the existence of the constants $\delta$, $c$, and $\alpha$ as established in the proof of part (a) of~\cite[Proposition~4.1]{Bernik-Kleinbock-Margulis-01:MR1829381}.  It remains for us to show that, given the definition of $r$ in~\eqref{slv2}, it is indeed possible to choose such constants in such a way that the relations~\eqref{tree}--\eqref{eq_4_5c}  hold.

With this in mind, set $$\delta:=\eta,$$  where $\eta$ is defined by~\eqref{def_r_V}. It follows from the definition of $\eta$ and the  well known bound  $N_d\le 5^d$ for the Besicovitch constant (cf.~Remark~\ref{besicocst} p.\pageref{besicocst})
that \eqref{tree} is satisfied with $\delta=\eta$.  We proceed with verifying~\eqref{eq_4_5a} and~\eqref{eq_4_5b}.

Regarding~\eqref{eq_4_5a}, let $\bm{g}=\left(\bm{u}_1 \cdot \vv f, \bm{u}_2  \cdot \vv f+u_0\right)\in\mathcal{G}$.  Also, let   $\bm{u}:=\left(\bm{u}_1, \bm{u}_2\right)$ with $\bm{u}_1,\bm{u}_2\in\Sph^{n-1}$  and  let $\bm{v}:=\left(v_1, v_2\right)\in\Sph^{1}$.  Furthermore, let $\bm{w}$ denote the vector $\bm{w}:=v_1\bm{u}_1+v_2\bm{u}_2$.  Since by the definition of $\mathcal{G}$, the vectors  $\bm{u}_1$ and  $\bm{u}_2$ are orthogonal, it  follows that  $\bm{w}\in\Sph^{n-1}$.  Now observe that
for any multi--index $\vv\beta$ such that $\left|\vv\beta\right|\le l$, $$\bm{v}\cdot\partial_{\vv\beta}\bm{g} = \bm{w}\cdot\partial_{\vv\beta}\bm{f}.$$
By the definition of $s_0=s(l; \bm f, \tcUp)$, there exists $\bm{s}\in\Sph^{d-1}$ and $1 \le k\le l$ such that
\[
\left|\bm{v}\cdotp\frac{\partial^k \bm{g}}{\partial\bm{s}^k}(\vv x_0) \right|\, = \, \left|\bm{w}\cdotp\frac{\partial^k \bm{f}}{\partial\bm{s}^k}(\vv x_0) \right|\, \ge \, s_0 \, .
\]
As per usual, $\vv x_0$  denotes here the centre of $\cU$. It follows that for any  $\bm{x}\in\tcUp$, we have that
\begin{align}
\left|\bm{v}\cdotp\frac{\partial^k \bm{g}}{\partial\bm{s}^k}(\bm{x}) \right|
&\ge \, \left|\bm{w}\cdotp\frac{\partial^k \bm{f}}{\partial\bm{s}^k}(\vv x_0) \right| - \left|\bm{w}\cdotp\left(\frac{\partial^k \bm{f}}{\partial\bm{s}^k}(\vv x_0) - \frac{\partial^k \bm{f}}{\partial\bm{s}^k}(\bm{x})\right)\right| \nonumber\\[2ex]
&\ge \, s_0 - \left\| \frac{\partial^k \bm{f}}{\partial\bm{s}^k}(\vv x_0) - \frac{\partial^k \bm{f}}{\partial\bm{s}^k}(\bm{x})\right\|_{2}.
\end{align}
The same arguments as those used to prove~\eqref{tech_corollary_one_lb_one} can be employed to  show that
\[
\left|\bm{v}\cdot\frac{\partial^k\bm{g}}{\partial\bm{s}^k}\left(\bm{x}\right) \right|\, \ge \, \frac{s_0}{2}  \qquad \forall \  \vv x\in\tcUp  \, .
\]
This  proves~\eqref{eq_4_5a} with
$$
c:= \frac{s_0}{2}\cdotp
$$

\noindent We now turn our attention to~\eqref{eq_4_5b}. With $ \vv g $ and $\vv u$ as above, first note that for any $\bm{x}\in\tcUp$ and for any multi--index $\vv\beta$ such that $|\beta|=l$, we have that
\begin{align} \label{ez_first}
\!\!\!\! \|\partial_{\vv\beta}\bm{g}\left(\bm{x}\right)  & -  \partial_{\vv\beta}\bm{g}\left(\vv x_0\right)\|_2 =  \left\|\left(\bm{u}_1\cdot\left(\partial_{\vv\beta}\vv f(\bm{x})-\partial_{\vv\beta}\vv f(\vv x_0)\right),\,\bm{u}_2\cdot\left(\partial_{\vv\beta}\vv f(\bm{x})-\partial_{\vv\beta}\vv f(\vv x_0)\right)\right)\right\|_2  \, .
\end{align}
Next, note that from  the Cauchy--Schwarz inequality, we have that, for $i=1,2$,
\begin{equation} \label{ez_Cauchy--Schwartz}
\left(\bm{u}_i\cdot\left(\partial_{\vv\beta}\vv f(\bm{x})-\partial_{\vv\beta}\vv f(\vv x_0)\right)\right)^2 \le \|\partial_{\vv\beta}\vv f(\bm{x})-\partial_{\vv\beta}\vv f(\vv x_0)\|_2^2.
\end{equation}
On combining  ~\eqref{ez_first} and ~\eqref{ez_Cauchy--Schwartz}, we find that
$$
\left\|\partial_{\vv\beta}\bm{g}\left(\bm{x}\right)-\partial_{\vv\beta}\bm{g}\left(\vv x_0\right)\right\|_2\leq \sqrt{2}\|\partial_{\vv\beta}\vv f(\bm{x})-\partial_{\vv\beta}\vv f(\vv x_0)\|_2.
$$
Now, since $\vv f$ satisfies Assumption 1, in view of\eqref{eqn:def_N}, we obtain that
$$
\|\partial_{\vv\beta}\vv f(\bm{x})-\partial_{\vv\beta}\vv f(\vv x_0)\|_2\leq M\sqrt{2} \left\|\bm{x}-\vv x_0\right\|_2 \, .
$$
Hence, for any $\bm{x}, \bm{y}\in\tcUp$ we have that
\begin{equation*} \label{ez_g}
\left\|\partial_{\vv\beta}\bm{g}\left(\bm{x}\right)-\partial_{\vv\beta}\bm{g}\left(\bm{y}\right)\right\|_2 \,\le \, 2M \left(\left\|\bm{x}-\vv x_0\right\|_2+\left\|\bm{y}-\vv x_0\right\|_2\right).
\end{equation*}
In view of~\eqref{slv2}, we also have that
\begin{eqnarray*}\label{ez_x}
\left\|\bm{x}-\vv x_0\right\|_2+\left\|\bm{y}-\vv x_0\right\|_2&\leq & 2\cdot 3^{n+d+2}\cdot r \nonumber \\[2ex] &\leq & 2\cdot 3^{n+d+2}\cdot\frac{\eta s_0}{4\cdot10^7 3^{n+d+2}\, d M l^{l+2}(l+1)!}  \nonumber \\[2ex]
&=  &   \frac{\eta s_0}{2\cdot10^7\, d M l^{l+2}(l+1)!}\cdotp
\end{eqnarray*}
The upshot is that
\begin{equation} \label{ez_f}
\underset{\bm{x},\bm{y}\in\tcUp}{\sup} \left\|\partial_{\vv\beta}\bm{g}\left(\bm{x}\right)-\partial_{\vv\beta}\bm{g}\left(\bm{y}\right)\right\|_2\, \le \frac{\eta s_0}{10^7\, d  l^{l+2}(l+1)!}
\end{equation}
for any  multi--index $\vv\beta$ with $\left|\vv\beta\right|=l$. This proves~\eqref{eq_4_5b} with $$\alpha:=\frac{32}{10^7\sqrt{d}}\, , $$ which clearly satisfies~\eqref{eq_4_5c}.


To prove part~(b) of Proposition~\ref{proposition_explicit_4_1}, we closely follow the proof of part (b) of  ~\cite[Proposition~4.1]{Bernik-Kleinbock-Margulis-01:MR1829381}.  The  new ingredient in our proof is the calculation of explicit constants at appropriate places. With this in mind, let
$\vv g = (g_1,g_2) \in\cG$ and take $B$ appearing at the start of the proof of ~\cite[Proposition~4.1(b)]{Bernik-Kleinbock-Margulis-01:MR1829381} to be $\cU$, so that $\hat B=\frac{1}{2} \vv U $.
We claim that there exists a point $\bm{y}\in \hat B$ such that
\begin{equation} \label{ez_gp}
\left\|\bm{g}(\bm{y})\right\|_2 \, \ge \, \tau:=\frac{r^l s_0}{4l^l(l+1)!}\cdotp
\end{equation}
To see that this is so, take $\bm{v}:=(1,0)\in\Sph^1$. In view of~\eqref{eq_4_5a}, there exists a vector $\bm{u}\in\Sph^{d-1}$ and $ 1 \leq k\leq l$ such that
$$
\left|\bm{v}\cdot\frac{\partial^k\bm{g}}{\partial\bm{u}^k}\left(\bm{x}\right)\right| \; =\; \left|\frac{\partial^k g_1}{\partial\bm{u}^k}\left(\bm{x}\right) \right| \; \ge \; c:=\frac{s_0}{2}  \quad \forall \  \bm{x}\in \cU \, .
$$
Thus, on applying ~\cite[Lemma~3.6]{Bernik-Kleinbock-Margulis-01:MR1829381} to the function $g_1$ and the ball $\hat B$, we obtain that
$$
\underset{\bm{x}, \bm{y}\in\hat B}{\sup} \left|g_1(\bm{x}) - g_1(\bm{y}) \right| \, \ge \, \frac{r^l c}{l^l(l+1)!}=\frac{r^l s_0}{2l^l(l+1)!}\cdotp
$$
This implies the existence of a point $\vv y\in \hat B$ such that
\begin{equation*} \label{ez_a}
\left\|\bm{g}(\bm{y})\right\|_2 \, \ge \, \left|g_1\left(\bm{y}\right)\right|\,\ge \, \tau
\end{equation*}
as claimed.  Next, observe that for any $\vv w \in\Sph^{d-1}$ and  any $\bm{x}\in\cU$,
$$
\frac{\partial \bm{g}}{\partial \bm{w}}(\bm{x}) \, = \,
\begin{pmatrix}
\bm{u}_1\cdot\frac{\partial \bm{f}}{\partial \bm{w}}(\bm{x})\\[2ex]
\bm{u}_2\cdot\frac{\partial \bm{f}}{\partial \bm{w}}(\bm{x})
\end{pmatrix}.
$$
Therefore, on  using the Cauchy--Schwarz inequality, we obtain via \eqref{eqn:def_N}  that

\begin{equation}\label{ez81}
\left\| \frac{\partial \bm{g}}{\partial \bm{w}}(\bm{x})\right\|_2 \, \le\, \sqrt{2}\left\|\frac{\partial \bm{f}}{\partial \bm{w}}(\bm{x})\right\|_2 \, \le \, \sqrt{2}M \, .
\end{equation}

Now, observe that, in view of \eqref{slv2}, of the definition of $\tau$   and of the fact that $s_0 \le M$, we have that
$$
\tau \le  r \, M.
$$
Consider the ball $B'\subset B=\cU$ with radius $\tau/(2M)  \leq r/2$ centred at $\vv y$, where $\vv y$ satisfies ~\eqref{ez_gp}.
 Take a vector $\vv v\in\Sph^1$ orthogonal to $\vv g(\vv y)$. In view of~\eqref{eq_4_5a}, there exists a vector $\bm{u}\in\Sph^{d-1}$ and $ 1 \leq k\leq l$ such that
$$
\left|\bm{v}\cdot\frac{\partial^k\bm{g}}{\partial\bm{u}^k}\left(\bm{x}\right)\right|  \; \ge \; c=\frac{s_0}{2}   \quad \forall \  \bm{x}\in\cU  \; .
$$
Thus,  on applying ~\cite[Lemma~3.6]{Bernik-Kleinbock-Margulis-01:MR1829381} to the function $ \vv x  \to  \vv v \cdot \vv g(\vv x)$ and the ball $B'$, we obtain that
\begin{equation} \label{ez_w}
\sup_{\vv x\in B'}|\vv v\cdot\vv g(\vv x)|\geq\frac{s_0}{4l^l(l+1)!}\left(\frac{\tau}{M}\right)^l.
\end{equation}

\noindent On the other hand, the upper bound~\eqref{ez81} implies that
\begin{equation} \label{ez_delta}
\sup_{\vv x\in B'}\|\vv g(\vv x)-\vv g(\vv y)\|_2\leq \, \frac{\tau}{2M} \,   \sqrt{2}M \,= \,  \frac{\tau}{\sqrt{2}}\cdotp
\end{equation}
The upshot of \eqref{ez_w} and \eqref{ez_delta} is that we are able to  apply ~\cite[Lemma~4.2]{Bernik-Kleinbock-Margulis-01:MR1829381} to  the map $\vv g:B'\rightarrow\R^2$ to yield ~\eqref{inegrhogb} and thereby complete the proof of  part~(b) of Proposition~\ref{proposition_explicit_4_1}. For ease of comparison, we point out that the quantities $a$, $\delta$ and $w$ appearing in the statement of ~\cite[Lemma~4.2]{Bernik-Kleinbock-Margulis-01:MR1829381}  correspond to  $\tau$, $\tau/\sqrt{2}$ and the right--hand side of\eqref{ez_w} respectively.

\end{proof}

\vspace{5ex}

 \noindent{\bf Acknowledgements.}  The  main catalysts for this paper were  the 'York-BT Mathematics of MIMO'  meeting at BT, Adastral Park on 14 June 2013 and the subsequent  international 'Workshop on interactions between number theory and wireless communication' at the University of York between 9-23 May 2014. The aim of these meetings was to explore potential applications to wireless technology of current research on Diophantine approximation. We would like to take this opportunity to thank the  `electronic' participants for putting up with our naive questions and often nonsensical ramblings.   In particular, we would like to thank to Alistar Burr (Dept. of Electronics, York) and Keith Briggs (BT Labs) for their enthusiastic support in making sure that the meetings actually happened -- they were our main link to the electronics world!

 All five authors of this paper are  number theorists and we have all benefited hugely from the presence of Maurice Dodson in our lives at some stage in our careers in some form or another.  Not only does he have a wonderful mind and personality but his intellectual curiosity knows no boundaries.  We would like to think that some of his curiosity has rubbed off onto us and for this we are forever in his debt.  There is so much to gain in engaging in dialogue with researchers in other disciplines and  attempting to overcome `language' barriers. In short, thank you Maurice for being a genuine intellectual.

SV would like to thank his teenage daughters Iona and Ayesha for still being a joyous addition in his life but would strongly urge them to improve their taste in music.  Also a massive thanks to Bridget Bennett for sticking around even at fifty  -- congratulations!
EZ is enormously grateful to Aljona who not only recently gave birth to their daughter Alyssa  but looked after her during his obsession with this project.  Infinite thanks to Alyssa herself,  just for her simple existence and for all the joy and happiness she has already brought to his life.

\end{document}